\documentclass[a4paper,10pt]{amsart}

\usepackage[T1]{fontenc}
\usepackage{lmodern}
\usepackage{amsmath}
\usepackage{amsfonts}
\usepackage{amssymb}
\usepackage{amsthm}
\usepackage{thmtools}
\usepackage{mathtools}
\usepackage{enumitem}
\usepackage{placeins}
\usepackage{multirow}
\usepackage[pdftex,
            pdfauthor={Sven Möller and Nils R. Scheithauer},
            pdftitle={A Geometric Classification of the Holomorphic Vertex Operator Algebras of Central Charge 24},
            pdfsubject={},
            pdfkeywords={},
            pdfproducer={},
            pdfcreator={}]{hyperref}
\usepackage{caption}

\theoremstyle{plain}
\newtheorem{thm}{Theorem}[section]
\newtheorem{prop}[thm]{Proposition}
\newtheorem{cor}[thm]{Corollary}
\newtheorem{lem}[thm]{Lemma}

\newtheorem*{thm*}{Theorem}
\newtheorem*{conj*}{Conjecture}
\newtheorem*{prop*}{Proposition}

\newtheoremstyle{narrow}
  {.5em} 
  {.5em} 
  {\itshape} 
  {} 
  {\bfseries} 
  {.} 
  {.5em} 
  {} 

\theoremstyle{narrow}

\newtheorem*{thn*}{Theorem}

\newtheorem*{conjn*}{Conjecture}

\theoremstyle{definition}
\newtheorem{defi}[thm]{Definition}

\newtheorem*{nota*}{Notation}

\newcommand{\Q}{\mathbb{Q}}
\newcommand{\Z}{\mathbb{Z}}
\newcommand{\Ns}{\mathbb{Z}_{>0}}
\newcommand{\N}{\mathbb{Z}_{\geq0}}
\newcommand{\C}{\mathbb{C}}
\newcommand{\R}{\mathbb{R}}
\renewcommand{\H}{\mathbb{H}}
\newcommand{\wt}{\operatorname{wt}}

\renewcommand{\i}{\mathrm{i}}
\newcommand{\e}{\mathrm{e}}

\newcommand{\Aut}{\operatorname{Aut}}
\newcommand{\Hom}{\operatorname{Hom}}

\newcommand{\rk}{\operatorname{rk}}

\newcommand{\voa}{vertex operator algebra}
\newcommand{\Voa}{Vertex operator algebra}
\newcommand{\VOA}{Vertex Operator Algebra}
\newcommand{\vosa}{vertex operator subalgebra}

\newcommand{\fpvosa}{fixed-point vertex operator subalgebra}

\newcommand{\vac}{\textbf{1}}

\newcommand{\id}{\operatorname{id}}
\newcommand{\amgis}{\zeta}
\newcommand{\eps}{\varepsilon}
\newcommand{\lcm}{\operatorname{lcm}}

\newcommand{\ee}{\mathfrak{e}}
\newcommand{\g}{\mathfrak{g}}
\newcommand{\h}{\mathfrak{h}}

\newcommand{\orb}{\operatorname{orb}}

\newcommand{\strathol}{strongly rational, holomorphic}
\newcommand{\strat}{strongly rational}

\renewcommand{\O}{\operatorname{O}}
\newcommand{\Co}{\operatorname{Co}}

\newcommand{\gdh}{generalised deep hole}

\newcommand{\GDH}{Generalised Deep Hole}

\newcommand{\no}{\,{\raise0.25em\hbox{$\mathop{\hphantom{\cdot}}\limits^{_{\circ}}_{^{\circ}}$}}\,}

\newcommand{\bpmod}{\!\mod}

\renewcommand{\arraystretch}{1.2}

\begin{document}

\title[Geometric Classification]{{A Geometric Classification of the\\Holomorphic Vertex Operator Algebras\\of Central Charge 24}}
\author[S.\ Möller and N.~R.\ Scheithauer]{Sven Möller\textsuperscript{\lowercase{a}} and Nils~R. Scheithauer\textsuperscript{\lowercase{b}}}
\thanks{\textsuperscript{a}{Research Institute for Mathematical Sciences, Kyoto University, Kyoto, Japan}}
\thanks{\textsuperscript{b}{Technische Universität Darmstadt, Darmstadt, Germany}}
\thanks{Email: \href{mailto:math@moeller-sven.de}{\nolinkurl{math@moeller-sven.de}}, \href{mailto:scheithauer@mathematik.tu-darmstadt.de}{\nolinkurl{scheithauer@mathematik.tu-darmstadt.de}}}

\begin{abstract}
We associate with a \gdh{} of the Leech lattice \voa{} a generalised hole diagram. We show that this Dynkin diagram determines the \gdh{} up to conjugacy and that there are exactly $70$ such diagrams. In an earlier work we proved a bijection between the \gdh{}s and the \strathol{} \voa{}s of central charge $24$ with non-trivial weight-$1$ space. Hence, we obtain a new, geometric classification of these \voa{}s, generalising the classification of the Niemeier lattices by their hole diagrams.
\end{abstract}

\maketitle

\setcounter{tocdepth}{1}
\tableofcontents


\section{Introduction}

In 1968, Niemeier classified the positive-definite, even, unimodular lattices of rank~$24$ \cite{Nie73}. He showed that up to isomorphism there are exactly $24$ such lattices and that the isomorphism class of each lattice is uniquely determined by its root system. The Leech lattice $\Lambda$ is the unique Niemeier lattice without roots. There are at least five proofs of this classification result. Niemeier applied Kneser's neighbourhood method. Venkov found a proof based on harmonic theta series \cite{Ven80}. It can also be derived from Conway, Parker and Sloane's classification of the deep holes of the Leech lattice \cite{CPS82,Bor85b} and from the Smith-Minkowski-Siegel mass formula \cite{CS82c,CS99}. Finally, it also follows from the classification of certain automorphic representations of $\O_{24}$ \cite{CL19}.

We describe the third proof in more detail. In \cite{Bor85b} Borcherds showed that the Leech lattice $\Lambda$ is the unique Niemeier lattice without roots (see also \cite{Con69}) and that the orbits of deep holes of $\Lambda$, i.e.\ points in $\Lambda\otimes_\Z\R$ that have maximal distance to $\Lambda$, are in natural bijection with the other Niemeier lattices. These results are proved without explicitly classifying the deep holes or the Niemeier lattices. In \cite{CPS82} Conway, Parker and Sloane associate with a deep hole in $\Lambda$ a hole diagram, a certain affine Dynkin diagram whose vertices are the closest lattice points, and classify the possible diagrams by geometric methods. They find $23$ diagrams and show that a deep hole is fixed up to equivalence by its hole diagram. This implies that there are exactly $23$ Niemeier lattices with roots. In this paper we generalise this approach to \strathol{} \voa{}s of central charge $24$.

\Voa{}s and their representations axiomatise $2$-dimensional conformal field theories \cite{Bor86,FLM88}. They have found various applications in mathematics and mathematical physics, e.g., in geometry, group theory and the theory of automorphic forms. The theory of these algebras is in certain aspects similar to the theory of even lattices over the integers.

The weight-$1$ subspace $V_1$ of a \strathol{} \voa{} $V$ of central charge $24$ is a reductive Lie algebra. In 1993, Schellekens \cite{Sch93} (see also \cite{EMS20a}) showed that there are at most $71$ possibilities for this Lie algebra using the theory of Jacobi forms. He conjectured that all potential Lie algebras are realised and that the $V_1$-structure fixes the \voa{} up to isomorphism. By the work of many authors over the past three decades the following result is now proved (see, e.g., \cite{LL20,LS20b} and the references cited therein).
\begin{thn*}
Up to isomorphism there are exactly $70$ \strathol{} \voa{}s $V$ of central charge $24$ with $V_1\neq\{0\}$. Such a \voa{} is uniquely determined by its $V_1$-structure.
\end{thn*}
The proof is based on a case-by-case analysis and uses a variety of methods.

The $24$ \voa{}s $V_N$ associated with the Niemeier lattices $N$ are examples of \voa{}s on Schellekens' list.

In this paper we give a uniform, geometric proof of the theorem based on the results in \cite{MS23}, which generalises the classification of the Niemeier lattices by enumeration of the corresponding deep holes of the Leech lattice $\Lambda$ \cite{CPS82,Bor85b}.

\medskip

One method to construct \voa{}s is the cyclic orbifold construction \cite{EMS20a}. Let $V$ be a \strathol{} \voa{} and $g$ an automorphism of $V$ of finite order~$n$. Then the fixed-point subalgebra $V^g$ is a \strat{} \voa{} with $n^2$ non-isomorphic irreducible modules, which can be realised as the eigenspaces of $g$ acting on the unique irreducible twisted modules $V(g^i)$ of $V$. If the twisted modules $V(g^i)$ have conformal weights in $(1/n)\Ns$ for $i\neq0\bpmod{n}$, then the direct sum $V^{\orb(g)}\coloneqq\oplus_{i\in\Z_n}V(g^i)^g$ is again a \strathol{} \voa{} of the same central charge as $V$. There is also an inverse orbifold construction, i.e.\ an automorphism $h$ of $V^{\orb(g)}$ such that $(V^{\orb(g)})^{\orb(h)}=V$.

Suppose that $V$ has central charge $24$ and that $n>1$. Pairing the character of $V^g$ with a certain vector-valued Eisenstein series of weight $2$ we obtain \cite{MS23}:
\begin{thn*}[Dimension Formula]
The dimension of the weight-$1$ subspace of $V^{\orb(g)}$ is given by
\begin{equation*}
\dim(V_1^{\orb(g)})=24+\sum_{d\mid n}c_n(d)\dim(V_1^{g^d})-R(g)
\end{equation*}
where the $c_n(d)\in\Q$ are defined by $\sum_{d\mid n}c_n(d)(t,d)=n/t$ for all $t\mid n$ and the remainder term $R(g)$ is non-negative. In particular,
\begin{equation*}
\dim(V_1^{\orb(g)})\leq24+\sum_{d\mid n}c_n(d)\dim(V_1^{g^d}).
\end{equation*}
\end{thn*}
The remainder term $R(g)$ is described explicitly. It depends on the dimensions of the weight spaces of the irreducible $V^g$-modules of weight less than 1.

The upper bound in the theorem motivates the following definition. The automorphism $g$ is called a \emph{\gdh} of $V$ if
\begin{enumerate}[topsep=3pt, partopsep=0pt]
\item the upper bound in the dimension formula is attained, i.e.\ $\dim(V_1^{\orb(g)})=24+\sum_{d\mid n}c_n(d)\dim(V_1^{g^d})$,
\item any Cartan subalgebra of $V_1^g$ is also a Cartan subalgebra of $V_1^{\orb(g)}$, i.e.\ $\rk(V_1^{\orb(g)})=\rk(V_1^g)$.
\end{enumerate}
We also call the identity automorphism a \gdh{}.

Let $V_\Lambda$ be the \voa{} of the Leech lattice $\Lambda$. Recall that algebraic conjugacy means conjugacy of cyclic subgroups. An averaged version of Kac's very strange formula implies \cite{MS23} (cf.\ \cite{ELMS21}):
\begin{thn*}[Holey Correspondence]
The orbifold construction $g\mapsto V_{\Lambda}^{\orb(g)}$ defines a bijection between the algebraic conjugacy classes of \gdh{}s $g$ in $\Aut(V_\Lambda)$ with $\rk((V_\Lambda^g)_1)>0$ and the isomorphism classes of \strathol{} \voa{}s $V$ of central charge $24$ with $V_1\neq\{0\}$.
\end{thn*}

Let $g\in\Aut(V_\Lambda)$ be a \gdh{}. Then $\h\coloneqq(V_\Lambda^g)_1$ is a Cartan subalgebra of $(V_{\Lambda}^{\orb(g)})_1$. It acts on $(V_{\Lambda}(g))_1$. The corresponding weights form a Dynkin diagram, which we denote by $\Phi(g)$. Then the \emph{generalised hole diagram} of $g$ is defined as the pair $(\varphi(g),\Phi(g))$, where $\varphi(g)$ denotes the cycle shape of the image of $g$ under the natural projection $\Aut(V_\Lambda) \to \O(\Lambda)$. For example, if $V_{\Lambda}^{\orb(g)}$ is isomorphic to the \voa{} $V_N$ of the Niemeier lattice with Dynkin diagram $N$, then the generalised hole diagram of $g$ is $(1^{24},{\tilde N})$ where ${\tilde N}$ is the affine Dynkin diagram corresponding to $N$, i.e.\ the hole diagram of the Niemeier lattice inside the Leech lattice $\Lambda$.

Our main result is the following (see \autoref{thm:gdhclass}):
\begin{thn*}[Classification of \GDH{}s]
There are exactly $70$ conjugacy classes of \gdh{}s $g$ in $\Aut(V_\Lambda)$ with $\rk((V_\Lambda^g)_1)>0$. The class of a \gdh{} is uniquely determined by its generalised hole diagram.
\end{thn*}

We outline the proof. The holey correspondence together with the lowest-order trace identity (see equation~\eqref{eq:221} in \autoref{sec:affine}) imply that there are at most $82$ possible generalised hole diagrams. These are described in \autoref{table:13} and \autoref{table:70}. Then, using geometric arguments similar to those by Conway, Parker and Sloane in \cite{CPS82} we reduce this number to $70$. On the other hand, in \cite{MS23} we explicitly list $70$ \gdh{}s with distinct diagrams. It follows that these are exactly the \gdh{}s $g$ of $V_\Lambda$ with $\rk((V_\Lambda^g)_1)>0$.

We observe (see \autoref{thm:proj11}):
\begin{thn*}[Projection to $\Co_0$]
Under the natural projection $\Aut(V_\Lambda)\to\O(\Lambda)$ the $70$ conjugacy classes of \gdh{}s $g$ with $\rk((V_\Lambda^g)_1)>0$ map to the $11$ conjugacy classes in $\O(\Lambda)\cong\Co_0$ with cycle shapes $1^{24}$, $1^82^8$, $1^63^6$, $2^{12}$, $1^42^24^4$, $1^45^4$, $1^22^23^26^2$, $1^37^3$, $1^22^14^18^2$, $2^36^3$ and $2^210^2$.
\end{thn*}
This recovers the decomposition of the Schellekens \voa{}s into $12$ families described by Höhn in \cite{Hoe17} (cf.\ \cite{HM23} in the fermionic case). The precise connection is explored in \cite{HM22}, Section~4.2.

A consequence of the classification of \gdh{}s is:
\begin{thn*}[Classification of \VOA{}s]
Up to isomorphism there are exactly $70$ \strathol{} \voa{}s $V$ of central charge $24$ with $V_1\neq\{0\}$. Such a \voa{} is uniquely determined by its $V_1$-structure.
\end{thn*}

In contrast to the previous proof, our argument is uniform and, except for the lowest-order trace identity, independent of Schellekens' results.

\subsection*{Outline}
In \autoref{sec:voaaut} we describe the orbifold construction, lattice \voa{}s and the automorphisms of the Leech lattice \voa{}.

In \autoref{sec:24} we recall some results on \strathol{} \voa{}s of central charge $24$, in particular the bijection with the \gdh{}s of the Leech lattice \voa{}.

In \autoref{sec:holediag} we associate a generalised hole diagram with a \gdh{} of the Leech lattice \voa{}.

In \autoref{sec:gdhclass} we finally use the generalised hole diagrams to classify the \gdh{}s of the Leech lattice \voa{}.

\subsection*{Acknowledgements}
The authors thank Tomoyuki Arakawa and Gerald Höhn for valuable discussions, and the referees for their helpful comments. Sven Möller was supported by a JSPS \emph{Postdoctoral Fellowship for Research in Japan} and by JSPS Grant-in-Aid KAKENHI 20F40018. Nils Scheithauer acknowledges support by the LOEWE research unit \emph{Uniformized Structures in Arithmetic and Geometry} and by the DFG through the CRC \emph{Geometry and Arithmetic of Uniformized Structures}, project number 444845124.


\section{\VOA{}s and Their Automorphisms}\label{sec:voaaut}

In this section we review the cyclic orbifold construction and describe the automorphisms of the Leech lattice \voa{} $V_\Lambda$.

\medskip

A \voa{} $V$ is called \emph{\strat{}} if it is rational (as defined, e.g., in \cite{DLM97}), $C_2$-cofinite (or lisse), self-contragredient (or self-dual) and of CFT-type. This already implies that $V$ is simple.

Moreover, a simple \voa{} $V$ is said to be \emph{holomorphic} if $V$ itself is the only irreducible $V$-module. The central charge of a \strathol{} \voa{} $V$ is necessarily a non-negative multiple of $8$.

Examples of \strat{} \voa{}s are those associated with positive-definite, even lattices. If the lattice is unimodular, then the associated \voa{} is holomorphic.


\subsection{Orbifold Construction}

The cyclic orbifold construction \cite{EMS20a,Moe16} is an important tool that can be used to construct new \voa{}s from known ones.

Let $V$ be a \strathol{} \voa{} and $G=\langle g\rangle$ a finite, cyclic group of automorphisms of $V$ of order $n$.

By \cite{DLM00} there is an up to isomorphism unique irreducible $g^i$-twisted $V$-module $V(g^i)$ for each $i\in\Z_n$. The uniqueness of $V(g^i)$ implies that there is a representation $\phi_i\colon G\to\Aut_\C(V(g^i))$ of $G$ on the vector space $V(g^i)$ such that
\begin{equation*}
\phi_i(g)Y_{V(g^i)}(v,x)\phi_i(g)^{-1}=Y_{V(g^i)}(g v,x)
\end{equation*}
for all $v\in V$, $i\in\Z_n$. This representation is unique up to an $n$-th root of unity. Denote the eigenspace of $\phi_i(g)$ in $V(g^i)$ corresponding to the eigenvalue $\e^{2\pi\i j/n}$ by $W^{(i,j)}$. On $V(g^0)=V$ we choose $\phi_0(g)=g$.

By \cite{DM97,Miy15,CM16,McR21} the \fpvosa{} $V^g=W^{(0,0)}$ is again \strat{}. It has exactly $n^2$ irreducible modules, namely the $W^{(i,j)}$, $i,j\in\Z_n$ \cite{MT04,DRX17}. One can further show that the conformal weight $\rho(V(g))$ of $V(g)$ is in $(1/n^2)\Z$, and we define the type $t\in\Z_n$ of $g$ by $t=n^2\rho(V(g))\bpmod{n}$.

Assume for simplicity that $g$ has type~$0$, i.e.\ that $\rho(V(g))\in(1/n)\Z$. Then it is possible to choose the representations $\phi_i$ such that the conformal weights satisfy
\begin{equation*}
\rho(W^{(i,j)})=\frac{ij}{n}\bpmod{1}
\end{equation*}
and $V^g$ has fusion rules
\begin{equation*}
W^{(i,j)}\boxtimes W^{(k,l)}\cong W^{(i+k,j+l)}
\end{equation*}
for all $i,j,k,l\in\Z_n$, i.e.\ the fusion ring of $V^g$ is the group ring $\C[\Z_n\times\Z_n]$ \cite{EMS20a}. In particular, all irreducible $V^g$-modules are simple currents.

In essence, the results in \cite{EMS20a} show that for cyclic $G\cong\Z_n$ and \strathol{} $V$ the module category of $V^G$ is the \emph{twisted group double} $\mathcal{D}_\omega(G)$ where the $3$-cocycle $[\omega]\in H^3(G,\C^\times)\cong\Z_n$ is determined by the type $t\in\Z_n$. This proves a special case of a conjecture in \cite{DVVV89,DPR90} stated for arbitrary finite $G$. The general case is proved in \cite{DNR21b}.

\medskip

In general, a simple \voa{} $V$ is said to satisfy the \emph{positivity condition} if the conformal weight $\rho(W)>0$ for any irreducible $V$-module $W\not\cong V$ and $\rho(V)=0$.

Now, if $V^g$ satisfies the positivity condition (it is shown in \cite{Moe18} that this condition is almost automatically satisfied if $V$ is \strat{}) and $g$ has type~$0$, then the direct sum of $V^g$-modules
\begin{equation*}
V^{\orb(g)}\coloneqq\bigoplus_{i\in\Z_n}W^{(i,0)}
\end{equation*}
admits the structure of a \strathol{} \voa{} of the same central charge as $V$ and is called \emph{orbifold construction} associated with $V$ and $g$ \cite{EMS20a}. Note that $\bigoplus_{j\in\Z_n}W^{(0,j)}$ gives back the original \voa{} $V$.

\medskip

We briefly describe the \emph{inverse} (or \emph{reverse}) orbifold construction \cite{EMS20a,LS19}. Suppose that the \strathol{} \voa{} $V^{\orb(g)}$ is obtained by an orbifold construction as described above. Then via $\amgis v\coloneqq\e^{2\pi\i j/n}v$ for $v\in W^{(j,0)}$ we define an automorphism $\amgis$ of $V^{\orb(g)}$ of order~$n$ and type~$0$, and the unique irreducible $\amgis^j$-twisted $V^{\orb(g)}$-module is given by $V^{\orb(g)}(\amgis^j)=\bigoplus_{i\in\Z_n}W^{(i,j)}$, $j\in\Z_n$. Then
\begin{equation*}
(V^{\orb(g)})^{\orb(\amgis)}=\bigoplus_{j\in\Z_n}W^{(0,j)}=V,
\end{equation*}
i.e.\ orbifolding with $\amgis$ is inverse to orbifolding with $g$.


\subsection{Automorphisms of the Leech Lattice \VOA{}}\label{sec:lat}

We describe lattice \voa{}s \cite{Bor86,FLM88}, the automorphism group of the Leech lattice \voa{} $V_\Lambda$ and in particular its conjugacy classes, which were determined in \cite{MS23}.

For a positive-definite, even lattice $L$ with bilinear form $\langle\cdot,\cdot\rangle\colon L\times L\to\Z$ the associated \voa{} is given by
\begin{equation*}
V_L=M(1)\otimes\C_\eps[L]
\end{equation*}
where $M(1)$ is the Heisenberg \voa{} of rank $\rk(L)$ associated with $\h_L=L\otimes_\Z\C$ and $\C_\eps[L]$ the twisted group algebra, i.e.\ the algebra with basis $\{\ee_\alpha\,|\,\alpha\in L\}$ and products $\ee_\alpha\ee_\beta=\eps(\alpha,\beta)\ee_{\alpha+\beta}$ where $\eps\colon L\times L\to\{\pm1\}$ is a 2-cocycle satisfying $\eps(\alpha,\beta)/\eps(\beta,\alpha)=(-1)^{\langle\alpha,\beta\rangle}$.

Let $\O(L)$ denote the orthogonal group (or automorphism group) of the lattice~$L$. For $\nu\in\O(L)$ and a function $\eta\colon L\to\{\pm1\}$ the map $\phi_\eta(\nu)$ acting on $\C_\eps[L]$ as $\phi_\eta(\nu)(\ee_\alpha)=\eta(\alpha)\ee_{\nu\alpha}$ for $\alpha\in L$ and as $\nu$ on $M(1)$ defines an automorphism of $V_L$ if and only if
\begin{equation*}
\frac{\eta(\alpha)\eta(\beta)}{\eta(\alpha+\beta)}=\frac{\eps(\alpha,\beta)}{\eps(\nu\alpha,\nu\beta)}
\end{equation*}
for all $\alpha,\beta\in L$. In this case, $\phi_\eta(\nu)$ is called a \emph{lift} of $\nu$ and all such automorphisms form the subgroup $\O(\hat{L})$ of $\Aut(V_L)$. There is a short exact sequence
\begin{equation*}
1\to\Hom(L,\{\pm1\})\to\O(\hat{L})\to\O(L)\to1
\end{equation*}
with the surjection $\O(\hat{L})\to\O(L)$ given by $\phi_\eta(\nu)\mapsto\nu$. The image of $\Hom(L,\{\pm1\})$ in $\O(\hat{L})$ is exactly the lifts of $\id\in\O(L)$.

If the restriction of $\eta$ to the fixed-point lattice $L^\nu$ is trivial, we call $\phi_\eta(\nu)$ a \emph{standard lift} of $\nu$. It is always possible to choose $\eta$ in this way \cite{Lep85}. It was proved in \cite{EMS20a} that all standard lifts of a given $\nu\in\O(L)$ are conjugate in $\Aut(V_L)$.

For any \voa{} $V$, $K\coloneqq\langle\{\e^{v_0}\,|\,v\in V_1\}\rangle$ defines a normal subgroup of $\Aut(V)$ called the \emph{inner automorphism group} of $V$. By \cite{DN99} the automorphism group of $V_L$ is of the form
\begin{equation*}
\Aut(V_L)=\O(\hat{L})\cdot K,
\end{equation*}
and moreover $\Hom(L,\{\pm1\})$ is a subgroup of $K\cap\O(\hat{L})$ and $\Aut(V_L)/K$ isomorphic to a quotient of $\O(L)$.

\medskip

In the following, we specialise to the Leech lattice $\Lambda$, the up to isomorphism unique unimodular, positive-definite, even lattice of rank $24$ without roots, i.e.\ without vectors of norm $2$. The automorphism group $\O(\Lambda)$ is Conway's group $\Co_0$. Since $(V_\Lambda)_1=\{h(-1)\otimes\ee_0\,|\,h\in\h_\Lambda\}\cong\h_\Lambda$ with $\h_\Lambda=\Lambda\otimes_\Z\C$, the inner automorphism group is given by
\begin{equation*}
K=\{\e^{h_0}\,|\,h\in\h_\Lambda\}
\end{equation*}
and is abelian. Because $K\cap\O(\hat{\Lambda})=\Hom(\Lambda,\{\pm1\})$ in the special case of the Leech lattice, there is a short exact sequence
\begin{equation*}
1\to K\to\Aut(V_\Lambda)\to\O(\Lambda)\to1.
\end{equation*}
Hence, every automorphism of $V_\Lambda$ is of the form
\begin{equation*}
\phi_\eta(\nu)\sigma_h
\end{equation*}
with a lift $\phi_\eta(\nu)$ of some $\nu\in\O(\Lambda)$ and with $\sigma_h=\e^{2\pi\i h_0}$ for some $h\in\h_\Lambda$. The surjection $\Aut(V_\Lambda)\to\O(\Lambda)$ in the short exact sequence is given by $\phi_\eta(\nu)\sigma_h\mapsto\nu$. It suffices to take a standard lift $\phi_\eta(\nu)$ of $\nu$ because any two lifts of $\nu$ only differ by a homomorphism $\Lambda\to\{\pm 1\}$, which can be absorbed into $\sigma_h$. Moreover, since $\sigma_h=\id$ if and only if $h\in\Lambda'=\Lambda$, it is enough to take $h\in\h_\Lambda/\Lambda$.

\medskip

We describe the conjugacy classes of $\Aut(V_\Lambda)$. For $\nu\in\O(\Lambda)$ let
\begin{equation*}
\pi_\nu=\frac{1}{|\nu|}\sum_{i=0}^{|\nu|-1}\!\nu^i
\end{equation*}
denote the projection from $\h_\Lambda$ onto the elements of $\h_\Lambda$ fixed by $\nu$. The automorphism $\phi_\eta(\nu)\sigma_h$ is conjugate to $\phi_\eta(\nu)\sigma_{\pi_\nu(h)}$ for any $h\in\h_\Lambda$, and $\phi_\eta(\nu)$ and $\sigma_{\pi_\nu(h)}$ commute.

In \cite{MS23} all automorphisms in $\Aut(V_\Lambda)$ were classified up to conjugacy. A similar result for arbitrary lattice \voa{}s was proved in \cite{HM22}.
\begin{prop}[\cite{MS23}]\label{prop:leechconj}
Let $Q\coloneqq\{(\nu,h)\,|\,\nu\in N,h\in H_\nu\}$ where
\begin{enumerate}
\item\label{item:rep1} $N$ is a set of representatives for the conjugacy classes in $\O(\Lambda)$,
\item\label{item:rep2} $H_\nu$ is a set of representatives for the orbits of the action of the centraliser $C_{\O(\Lambda)}(\nu)$ on $\pi_\nu(\h_\Lambda)/\pi_\nu(\Lambda)$.
\end{enumerate}
Fix a section $\nu\mapsto\phi_\eta(\nu)$. Then the map $(\nu,h)\mapsto\phi_\eta(\nu)\sigma_h$ is a bijection from the set $Q$ to the conjugacy classes of $\Aut(V_\Lambda)$.
\end{prop}
Since $h\in\pi_\nu(\h_\Lambda)$, $\phi_\eta(\nu)$ and $\sigma_h$ commute. The automorphism $\phi_\eta(\nu)\sigma_h$ in $\Aut(V_\Lambda)$ has finite order if and only if $h$ is in $\pi_\nu(\Lambda\otimes_\Z\Q)$.

\medskip

We also describe the conjugacy classes in $\Aut(V_\Lambda)$ of a given finite order~$n$. First note that a standard lift $\phi_\eta(\nu)$ of $\nu$ has order $m =|\nu|$ if $m$ is odd or if $m$ is even and $\langle\alpha,\nu^{m/2}\alpha\rangle\in 2\Z$ for all $\alpha\in\Lambda$, and order $2m$ otherwise. In the latter case we say that $\nu$ exhibits \emph{order doubling}. Then $\phi_\eta(\nu)^m\ee_\alpha=(-1)^{m\langle\pi_\nu(\alpha),\pi_\nu(\alpha)\rangle}\ee_\alpha=(-1)^{\langle\alpha,\nu^{m/2}\alpha\rangle}\ee_\alpha$ for all $\alpha\in\Lambda$. Note that the map sending $\alpha$ to $m\langle\pi_\nu(\alpha),\pi_\nu(\alpha)\rangle=\langle\alpha,\nu^{m/2}\alpha\rangle \bpmod 2$ defines a homomorphism $\Lambda\to\Z_2$.

Let $\phi_\eta(\nu)$ be a standard lift. If $\nu$ exhibits order doubling, then there exists a vector $s_\nu\in(1/2m)\Lambda^\nu$ defining an inner automorphism $\sigma_{s_\nu}=\e^{2\pi\i(s_\nu)_0}$ of order $2m$ such that $\phi_\eta(\nu)\sigma_{s_\nu}$ has order $m$. If $\nu$ does not exhibit order doubling, we set $s_\nu=0$. Then the order of an automorphism $\phi_\eta(\nu)\sigma_{s_{\nu}+f}$ for $f\in\Lambda\otimes_\Z\Q$ is given by $\lcm(m,k)$ where $k$ is the smallest positive integer such that $kf$ is in $\Lambda$ or equivalently in the fixed-point lattice $\Lambda^\nu$.

For convenience, we define the $s_\nu$-shifted action of $C_{\O(\Lambda)}(\nu)$ on $\pi_\nu(\h_\Lambda)$ by
\begin{equation*}
\tau.f=\tau f+(\tau-\id)s_\nu
\end{equation*}
for all $\tau\in C_{\O(\Lambda)}(\nu)$ and $f\in\pi_\nu(\h_\Lambda)$. Then:
\begin{prop}[\cite{ELMS21}]\label{prop:leechconj2}
Fix a section $\nu\mapsto\phi_\eta(\nu)$ mapping only to standard lifts. A complete system of representatives for the conjugacy classes of automorphisms in $\Aut(V_\Lambda)$ of order $n$ is given by the $\phi_\eta(\nu)\sigma_{s_{\nu}+f}$ where
\begin{enumerate}
\item $\nu$ is from the representatives in $N\subseteq\O(\Lambda)$ of order $m$ dividing $n$,
\item $f$ is from the orbit representatives of the $s_\nu$-shifted action of the centraliser $C_{\O(\Lambda)}(\nu)$ on $(\Lambda^\nu/n)/\pi_{\nu}(\Lambda)$
\end{enumerate}
such that $\lcm(m,|\sigma_f|)=n$.
\end{prop}

\medskip

We conclude this section by recalling some results on the twisted modules of lattice \voa{}s. For a standard lift $\phi_\eta(\nu)$ the irreducible $\phi_\eta(\nu)$-twisted modules of a lattice \voa{} $V_L$ are described in \cite{DL96,BK04}. Together with the results in \cite{Li96} this allows us to describe the irreducible $g$-twisted $V_L$-modules for all finite-order automorphisms $g\in\Aut(V_L)$.

For simplicity, let $L$ be unimodular. Then $V_L$ is holomorphic and there is a unique irreducible $g$-twisted $V_L$-module $V_L(g)$ for each $g\in\Aut(V_L)$ of finite order. Let $g=\phi_\eta(\nu)\sigma_h$ for some standard lift $\phi_\eta(\nu)$ and $\sigma_h=\e^{2\pi\i h_0}$ for some $h\in\pi_\nu(L\otimes_\Z\Q)$. Then
\begin{equation*}
V_L(g)=M(1)[\nu]\otimes\C[-h+\pi_\nu(L)]\otimes\C^{d(\nu)}
\end{equation*}
with the twisted Heisenberg module $M(1)[\nu]$, grading by the lattice coset $-h+\pi_\nu(L)$ and defect $d(\nu)\in\Ns$. (The minus sign in $-h+\pi_{\nu}(\Lambda)$ has to do with the sign convention in the definition of twisted modules. Here, we follow the convention in, e.g., \cite{DLM00} as opposed to some older texts.)

Assume that $\nu$ has order~$m$ and cycle shape $\prod_{t\mid m}t^{b_t}$ with $b_t\in\Z$, i.e.\ the extension of $\nu$ to $\h_L$ has characteristic polynomial $\prod_{t\mid m}(x^t-1)^{b_t}$. Then the conformal weight of $V_L(g)$ is given by
\begin{equation*}
\rho(V_L(g))=\frac{1}{24}\sum_{t\mid m}b_t\left(t-\frac{1}{t}\right)+\min_{\alpha\in-h+\pi_{\nu}(L)}\frac{\langle\alpha,\alpha\rangle}{2}\geq0,
\end{equation*}
where $\rho_\nu =\frac{1}{24}\sum_{t\mid m}b_t\left(t-\frac{1}{t}\right)$ is called the \emph{vacuum anomaly} of $V_L(g)$ \cite{DL96}. Note that $\rho_\nu$ is positive for $\nu\neq\id$. The second term is half of the norm of a shortest vector in the lattice coset $-h+\pi_\nu(L)$.


\section{Holomorphic \VOA{}s of Central Charge 24}\label{sec:24}

In this section we recall the notion of the affine structure of a \strathol{} \voa{} of central charge $24$ and describe the bijection between these \voa{}s and the \gdh{}s of the Leech lattice \voa{} \cite{MS23}.


\subsection{Affine Structure}\label{sec:affine}

Let $V=\bigoplus_{n=0}^\infty V_n$ be a \voa{} of CFT-type. Then the zero modes
\begin{equation*}
[a,b]\coloneqq a_0b
\end{equation*}
for $a,b\in V_1$ endow the weight-$1$ space $V_1$ with the structure of a finite-dimensional Lie algebra. Moreover, the zero modes $a_0$ for $a\in V_1$ equip each $V$-module with an action of this Lie algebra.

If $g\in\Aut(V)$ is an automorphism of the \voa{} $V$, fixing the vacuum vector $\vac\in V_0$ and the Virasoro vector $\omega\in V_2$ by definition, then the restriction of $g$ to $V_1$ is a Lie algebra automorphism, possibly of smaller order.

If $V$ is also self-contragredient, then there exists a non-degenerate, invariant bilinear form $\langle\cdot,\cdot\rangle$ on $V$, which is unique up to a non-zero scalar and symmetric \cite{FHL93,Li94}. We normalise this bilinear form such that $\langle\vac,\vac\rangle=-1$. Then $a_1b=b_1a=\langle a,b\rangle\vac$ for all for $a,b\in V_1$.

\medskip

Let $\g$ be a simple, finite-dimensional Lie algebra with the non-degenerate, invariant bilinear form $(\cdot,\cdot)$ normalised such that $(\alpha,\alpha)=2$ for all long roots $\alpha$. The affine Kac-Moody algebra $\hat\g$ associated with $\g$ is the Lie algebra $\hat\g\coloneqq\g\otimes\C[t,t^{-1}]\oplus\C K$ with central element $K$ and Lie bracket
\begin{equation*}
[a\otimes t^m,b\otimes t^n]\coloneqq[a,b]\otimes t^{m+n}+m(a,b)\delta_{m+n,0}K
\end{equation*}
for $a,b\in\g$, $m,n\in\Z$.

A $\hat\g$-module is said to have level $k\in\C$ if $K$ acts as $k\id$. Let $\lambda\in P_+$ be a dominant integral weight and $k\in\C$. Then we denote by $L_{\hat\g}(k,\lambda)$ the irreducible quotient of the $\hat\g$-module of level $k$ induced from the irreducible highest-weight $\g$-module $L_\g(\lambda)$ (see, e.g., \cite{Kac90}).

For a positive integer $k\in\Ns$, $L_{\hat\g}(k,0)$ admits the structure of a rational \voa{}, called the simple affine \voa{} of level $k$, whose irreducible modules are given by the modules $L_{\hat\g}(k,\lambda)$ for $\lambda\in P_+^k$, the subset of the dominant integral weights $P_+$ of level at most $k$ \cite{FZ92}.

If $V$ is a self-contragredient \voa{} of CFT-type, the commutator formula implies that the modes satisfy
\begin{equation*}
[a_m,b_n]=(a_0b)_{m+n}+m(a_1b)_{m+n-1}=[a,b]_{m+n}+m\langle a,b\rangle\delta_{m+n,0}\id_V
\end{equation*}
for all $a,b\in V_1$, $m,n\in\Z$. Comparing this with the definition above we see that for a simple Lie subalgebra $\g$ of $V_1$ the map $a\otimes t^n\mapsto a_n$ for $a\in\g$ and $n\in\Z$ defines a representation of $\hat\g$ on $V$ of some level $k_\g\in\C$ with $\langle\cdot,\cdot\rangle|_\g=k_\g(\cdot,\cdot)$.

\medskip

Suppose that $V$ is \strat{}. Then it is shown in \cite{DM04b} that the Lie algebra $V_1$ is reductive, i.e.\ a direct sum of a semisimple and an abelian Lie algebra. Moreover, it is stated in \cite{DM06b} that for a simple Lie subalgebra $\g$ of $V_1$ the restriction of $\langle\cdot,\cdot\rangle$ to $\g$ is non-degenerate, the level $k_\g$ is a positive integer, the \vosa{} of $V$ generated by $\g$ is isomorphic to $L_{\hat\g}(k_\g,0)$ and $V$ is an integrable $\hat\g$-module.

Assume in addition that $V$ is holomorphic and of central charge $24$. Then the Lie algebra $V_1$ is zero, abelian of dimension $24$ or semisimple of rank at most $24$ \cite{DM04}. If the Lie algebra $V_1$ is semisimple, then it decomposes into a direct sum
\begin{equation*}
V_1\cong\g_1\oplus\ldots\oplus\g_r
\end{equation*}
of simple ideals $\g_i$ and the \vosa{} $\langle V_1\rangle$ of $V$ generated by $V_1$ is isomorphic to the tensor product of simple affine \voa{}s
\begin{equation*}
\langle V_1\rangle\cong L_{\hat\g_1}(k_1,0)\otimes\ldots\otimes L_{\hat\g_r}(k_r,0)
\end{equation*}
with levels $k_i\coloneqq k_{\g_i}\in\Ns$ and has the same Virasoro vector as $V$. The tensor-product decomposition of the \voa{} $\langle V_1\rangle$ is called the \emph{affine structure} of $V$ and denoted by $\g_{1,k_1}\ldots\g_{r,k_r}$.

Since $\langle V_1\rangle\cong L_{\hat\g_1}(k_1,0)\otimes\ldots\otimes L_{\hat\g_r}(k_r,0)$ is rational, $V$ decomposes into the direct sum of finitely many irreducible $\langle V_1\rangle$-modules
\begin{equation*}
V\cong\bigoplus_\lambda m_\lambda L_{\hat\g_1}(k_1,\lambda_1)\otimes\ldots\otimes L_{\hat\g_r}(k_r,\lambda_r)
\end{equation*}
where $m_\lambda\in\N$ and the sum ranges over finitely many $\lambda=(\lambda_1,\ldots,\lambda_r)$ with dominant integral weights $\lambda_i\in P_+^{k_i}(\g_i)$, i.e.\ of level at most $k_i$.

\medskip

Let $h_i^\vee$ denote the dual Coxeter number of $\g_i$. The fact that the character of $V$ is a Jacobi form of lattice index implies the trace identity
\begin{equation}\label{eq:221}\tag{S}
\frac{h_i^\vee}{k_i}=\frac{\dim(V_1)-24}{24}
\end{equation}
for all $i=1,\ldots,r$ (see \cite{Sch93,DM04,EMS20a}). As a consequence, the Lie algebra $V_1$ uniquely determines the affine structure, i.e.\ the levels $k_i$. The equation has exactly $221$ solutions (see Table~3 in \cite{ELMS21}).

In \cite{Sch93} Schellekens also derived so-called higher-order trace identities (cf.\ \cite{EMS20a}, Theorem~6.1), which allowed him to reduce the above $221$ affine structures down to $69$ by solving large integer linear programming problems on the computer. Together with the zero Lie algebra and the $24$-dimensional abelian Lie algebra this gives Schellekens' list of $71$ Lie algebras (see \autoref{table:70}) that occur as the weight-$1$ space of a \strathol{} \voa{} of central charge $24$ \cite{Sch93}.

\medskip

We shall however not make use of Schellekens' classification result but give an independent proof based on the classification of certain geometric structures in the Leech lattice $\Lambda$.


\subsection{\GDH{}s}

One of the main results of \cite{MS23} is a dimension formula for the weight-$1$ space of the cyclic orbifold construction $V^{\orb(g)}$.
\begin{thm}[Dimension Formula, \cite{MS23}, Theorem~5.3 and Corollary~5.7]
Let $V$ be a \strathol{} \voa{} of central charge $24$ and $g$ an automorphism of $V$ of finite order~$n>1$ and type~$0$ such that $V^g$ satisfies the positivity condition. Then the dimension of the weight-$1$ subspace of $V^{\orb(g)}$ is
\begin{equation*}
\dim(V_1^{\orb(g)})=24+\sum_{d\mid n}c_n(d)\dim(V_1^{g^d})-R(g)
\end{equation*}
where the $c_n(d)\in\Q$ are defined by $\sum_{d\mid n}c_n(d)(t,d)=n/t$ for all $t\mid n$ and the remainder term $R(g)$ is non-negative. In particular,
\begin{equation*}
\dim(V_1^{\orb(g)})\leq24+\sum_{d\mid n}c_n(d)\dim(V_1^{g^d}).
\end{equation*}
\end{thm}

This dimension formula is obtained by pairing the vector-valued character of the \fpvosa{} $V^g$ with a vector-valued Eisenstein series of weight $2$, and it generalises earlier results in \cite{Mon94,LS19,Moe16,EMS20b} under the assumption that the modular curve $\Gamma_0(n)\backslash\H^*$ has genus zero.

We point out that the upper bound in the dimension formula depends only on the action of $g$ on the weight-$1$ Lie algebra $V_1$.

An automorphism $g$ such that $\dim(V^{\orb(g)}_1)$ attains the above upper bound is called \emph{extremal}. We also call the identity automorphism extremal.

The upper bound in the dimension formula motivates the following definition.
\begin{defi}[\GDH{}, \cite{MS23}]\label{def:gendeephole}
Let $V$ be a \strathol{} \voa{} of central charge $24$ and $g\in\Aut(V)$ of finite order $n>1$. Suppose $g$ has type~$0$ and $V^g$ satisfies the positivity condition. Then $g$ is called a \emph{\gdh} of $V$ if
\begin{enumerate}
\item $g$ is extremal, i.e.\ $\dim(V_1^{\orb(g)})=24+\sum_{d\mid n}c_n(d)\dim(V_1^{g^d})$,
\item $\rk(V_1^{\orb(g)})=\rk(V_1^g)$.
\end{enumerate}
\end{defi}
In other words, we demand the dimension of the Lie algebra $V_1^{\orb(g)}$ to be maximal with respect to the upper bound from the dimension formula and the rank to be minimal with respect to the obvious lower bound $\rk(V_1^g)$.

By convention, we call the identity automorphism a \gdh{}.

Recall that the Lie algebras $V_1^g$ and $V_1^{\orb(g)}$ are reductive. By Lemma~8.1 in \cite{Kac90} the centraliser in $V_1^{\orb(g)}$ of any choice of Cartan subalgebra of $V_1^g$ is a Cartan subalgebra of $V_1^{\orb(g)}$. Condition~$(2)$ is hence equivalent to demanding that any Cartan subalgebra of $V_1^g$ also be a Cartan subalgebra of $V_1^{\orb(g)}$. It can be replaced by the equivalent condition that the inverse-orbifold automorphism restricts to an inner automorphism on $V_1^{\orb(g)}$.

If $V\cong V_\Lambda$, the \voa{} associated with the Leech lattice $\Lambda$, then the rank condition is equivalent to demanding that $(V_\Lambda^g)_1$, which as a subalgebra of $(V_\Lambda)_1$ is abelian, be a Cartan subalgebra of $(V_\Lambda^{\orb(g)})_1$.

The second main result of \cite{MS23} is a natural bijection between the \gdh{}s of the Leech lattice \voa{} $V_\Lambda$ and the \strathol{} \voa{}s of central charge $24$ with non-vanishing weight-$1$ space.
\begin{thm}[Holey Correspondence, \cite{MS23}]\label{thm:bij}
The cyclic orbifold construction $g\mapsto V_\Lambda^{\orb(g)}$ defines a bijection between the algebraic conjugacy classes of \gdh{}s $g\in\Aut(V_\Lambda)$ with $\rk((V_\Lambda^g)_1)>0$ and the isomorphism classes of \strathol{} \voa{}s $V$ of central charge $24$ with $V_1\neq\{0\}$.
\end{thm}

The proof combines the dimension formula with an averaged version of Kac's very strange formula \cite{Kac90}. It does not use any classification result for either side of the correspondence.

This theorem generalises the natural bijection between the deep holes of the Leech lattice $\Lambda$ and the Niemeier lattices with roots \cite{Bor85b}, which is mediated by the holey construction \cite{CS82}.

Recall that the weight-$1$ Lie algebra $V_1$ of a \strathol{} \voa{} $V$ of central charge $24$ is either abelian or semisimple. If $g$ is a \gdh{} of $V_\Lambda$ with non-zero $(V_\Lambda^g)_1$ and $V\cong V_\Lambda^{\orb(g)}$, then $V_1$ is abelian if and only if $\dim(V_1)=24$ if and only if $V\cong V_\Lambda$ if and only if $g=\id$.

\medskip

The inverse orbifold construction corresponding to a \gdh{} $g$ of the Leech lattice \voa{} $V_\Lambda$ takes a very simple form \cite{ELMS21}. Assume that $V=V_\Lambda^{\orb(g)}$ is a \strathol{} \voa{} $V$ of central charge $24$ with $V_1=\g_1\oplus\ldots\oplus\g_r$ semisimple. Then the inverse-orbifold automorphism of $g$ (which must be of type~$0$ and extremal \cite{MS23}) is given by the inner automorphism
\begin{equation*}
\sigma_u=\e^{2\pi\i u_0}\quad\text{with}\quad u\coloneqq\sum_{i=1}^r\rho_i/h_i^\vee
\end{equation*}
where $h_i^\vee$ is the dual Coxeter number and $\rho_i$ the Weyl vector of $\g_i$. The order of $\sigma_u$ on each simple ideal $\g_i$ is $l_ih_i^\vee$ where $l_i\in\{1,2,3\}$ is the lacing number of $\g_i$. Hence, the order on $V_1$ is $\lcm(\{l_ih_i^\vee\}_{i=1}^r)$, which can be shown to equal the order $n$ of $\sigma_u$ on the whole \voa{} $V$. Of course, this equals the order of the corresponding \gdh{} $g\in\Aut(V_\Lambda)$.


\section{Generalised Hole Diagrams}\label{sec:holediag}

In this section we associate generalised hole diagrams with automorphisms of the Leech lattice \voa{} $V_{\Lambda}$. They will be the main datum we use to classify the \gdh{}s in $\Aut(V_{\Lambda})$.

\medskip

Let $V_{\Lambda}$ be the Leech lattice \voa{} and $g\in\Aut(V_{\Lambda})$ of order \mbox{$n>1$} and type~$0$ such that $V_{\Lambda}^g$ satisfies the positivity condition. Let $\nu$ be the projection of $g$ to $\O(\Lambda)$. Consider the orbifold construction $V_{\Lambda}^{\orb(g)}=\bigoplus_{i\in\Z_n}W_{\Lambda}^{(i,0)}$ and assume that $\rk((V_{\Lambda}^{\orb(g)})_1)=\rk((V_{\Lambda}^g)_1)>0$. Then $\g\coloneqq(V_{\Lambda}^{\orb(g)})_1$ is a semisimple or abelian Lie algebra and $\h=(V_\Lambda^g)_1=\{h(-1)\otimes\ee_0\,|\,h\in\pi_\nu(\h_\Lambda)\}$ is a Cartan subalgebra of $\g$.

The non-degenerate, invariant bilinear form $\langle\cdot,\cdot\rangle$ on $V_{\Lambda}^{\orb(g)}$, normalised such that $\langle\vac,\vac\rangle=-1$, restricts to a non-degenerate, invariant bilinear form on $\g$. The Cartan subalgebra $\h$ with the form $\langle\cdot,\cdot\rangle$ is naturally isometric to the subspace $\pi_\nu(\h_\Lambda)$ of $\h_\Lambda=\C\otimes_\Z\Lambda$. We may also identify $\h$ with $\h^*$ via $\langle\cdot,\cdot\rangle$. We write the Cartan decomposition corresponding to $\h$ as
\begin{equation*}
\g=\h\oplus\bigoplus_{\alpha\in\Phi}\g_{\alpha}
\end{equation*}
with root system $\Phi\subseteq\h^*$, which is empty if and only if $\g$ is abelian. The inverse orbifold automorphism $\amgis$ of $g$ restricts to an inner automorphism of $\g$, and $\g$ decomposes into eigenspaces
\begin{equation*}
\g=\g_{(0)}\oplus\g_{(1)}\oplus\ldots\oplus\g_{(n-1)}
\end{equation*}
where $\g_{(i)}=\g\cap W_{\Lambda}^{(i,0)}=(W_{\Lambda}^{(i,0)})_1$ and $\g_{(0)}=(V_\Lambda^g)_1=\h$. Since the action of $\amgis$ commutes with the adjoint action of $\h$ on $\g$ and the spaces $\g_{\alpha}$ are $1$-dimensional, each $\g_{\alpha}$ lies in exactly one $\g_{(i)}$. Hence, the root system $\Phi$ is a disjoint union
\begin{equation*}
\Phi=\Phi_{(1)}\cup\ldots\cup\Phi_{(n-1)}
\end{equation*}
with $\Phi_{(i)}=\{\alpha\in\Phi\,|\,\g_\alpha\subseteq\g_{(i)}\}$. We define
\begin{equation*}
\Pi(g)\coloneqq\Phi_{(1)}.
\end{equation*}
Since $(1,n)=1$, the weight-$1$ subspace of the irreducible $g$-twisted $V$-module $V_{\Lambda}(g)$ is $(W_{\Lambda}^{(1,0)})_1$. Hence, $\Pi(g)\subseteq\h^*$ can also be defined as the set of weights of the adjoint action of $\h$ on $V_{\Lambda}(g)_1$.
\begin{prop}
Assume that the root system $\Phi$ of $\g$ is non-empty. Then the inner products $2\langle\alpha_i,\alpha_j\rangle/\langle\alpha_i,\alpha_i\rangle$ for $\alpha_i,\alpha_j\in\Pi(g)$ form a generalised Cartan matrix with Dynkin diagram $\Phi(g)$ given by a subdiagram of the extended affine Dynkin diagram associated with the (finite) Dynkin diagram of $\Phi$.
\end{prop}
\begin{proof}
This is Proposition~8.6~c) in \cite{Kac90} together with the fact that $\g_{(0)}=\h$ is a Cartan subalgebra of $\g$.
\end{proof}

Let $\varphi(g)$ be the cycle shape of the image $\nu$ of $g$ under the natural projection $\Aut(V_{\Lambda}) \to \O(\Lambda)$. We define the \emph{generalised hole diagram} of $g$ as the pair
\begin{equation*}
(\varphi(g),\Phi(g)).
\end{equation*}
Note that the generalised hole diagram only depends on the algebraic conjugacy class of $g$ in $\Aut(V_{\Lambda})$. For $g=\id$ we set $\Phi(g)=\emptyset$.
\enlargethispage{\baselineskip}

\medskip

In the following, we study the weights $\Pi(g)\subseteq\h^*$ and the corresponding Dynkin diagram $\Phi(g)$ in more detail. These results will allow us to classify these diagrams in \autoref{sec:gdhclass} in the case where $g$ is a \gdh{}.

By \autoref{prop:leechconj}, up to conjugacy, $g=\phi_\eta(\nu)\sigma_h$ for some standard lift $\phi_\eta(\nu)$ of $\nu\in\O(\Lambda)\cong\Co_0$ and some $h\in\pi_\nu(\Lambda\otimes_\Z\Q)$. Denote $m=|\nu|$, which divides $n=|g|$, and let $\prod_{t\mid m}t^{b_t}$ be the cycle shape of $\nu$. Recall that the unique irreducible $g$-twisted $V_\Lambda$-module is of the form
\begin{equation*}
V_\Lambda(g)=M(1)[\nu]\otimes\C[-h+\pi_{\nu}(\Lambda)]\otimes\C^{d(\nu)}
\end{equation*}
with the twisted Heisenberg algebra $M(1)[\nu]$ and defect $d(\nu)\in\Ns$. $V_\Lambda(g)$ is spanned by the vectors
\begin{equation*}
v=h_1(-n_1)\ldots h_r(-n_r)\otimes\ee_\alpha\otimes t
\end{equation*}
where the $h_i$ are in certain eigenspaces of $\h_\Lambda$, $n_i\in(1/m)\Ns$, $\alpha\in-h+\pi_{\nu}(\Lambda)$ and $t\in\C^{d(\nu)}$. Such a vector has $L_0$-weight
\begin{equation*}
\wt(v)=\rho_\nu+n_1+\ldots+n_r+\frac{\langle\alpha,\alpha\rangle}{2}
\end{equation*}
with vacuum anomaly $\rho_\nu=\frac{1}{24}\sum_{t\mid m}b_t(t-1/t)$ and is acted on by the Cartan subalgebra $\h=(V_\Lambda^g)_1\cong\pi_\nu(\h_\Lambda)$ of $\g=(V_\Lambda^{\orb(g)})_1$ as
\begin{equation*}
h_0v=\langle h,\alpha\rangle v
\end{equation*}
for $h\in\pi_\nu(\h_\Lambda)$.

\begin{prop}\label{prop:leechroots}
The weights of the action of $\h$ on $V_{\Lambda}(g)_1$ are given by
\begin{equation*}
\Pi(g)=\big\{\alpha\in-h+\pi_{\nu}(\Lambda)\,\big|\,\langle\alpha,\alpha\rangle/2=1-\rho_\nu\big\}
\end{equation*}
if $d(\nu)=1$ and
\begin{equation*}
\Pi(g)=\emptyset
\end{equation*}
if $d(\nu)>1$.
\end{prop}
\begin{proof}
First, note that $\rho_\nu\geq1-1/m$ for all $\nu\in\O(\Lambda)$. Moreover, $V_\Lambda(g)_1$ cannot contain any vector $v=\ldots\otimes\ee_0\otimes t$ as this would lie in the centraliser of $\h=(V_\Lambda^g)_1$, contradicting the assumption that $\h$ is a Cartan subalgebra of $\g$. Therefore, a vector $v\in V_\Lambda(g)_1$ must be of the form $v=1\otimes\ee_\alpha\otimes t$ for some (non-zero) $\alpha\in-h+\pi_{\nu}(\Lambda)$ and $t\in\C^{d(\nu)}$, i.e.\ there can be no contribution to the weight from the twisted Heisenberg algebra except for the vacuum anomaly. Hence,
\begin{equation*}
V_\Lambda(g)_1=\big\{1\otimes\ee_\alpha\otimes t\,\big|\,\alpha\in-h+\pi_{\nu}(\Lambda)\text{ s.t. }\langle\alpha,\alpha\rangle/2=1-\rho_\nu,\;t\in\C^{d(\nu)}\big\}.
\end{equation*}
Since the action of the Cartan subalgebra $\h$ is independent of $t$ and all weight spaces are $1$-dimensional, either $d(\nu)=1$ or $V_\Lambda(g)_1=\{0\}$. In the first case
\begin{equation*}
\Pi(g)=\big\{\alpha\in-h+\pi_{\nu}(\Lambda)\,\big|\,\langle\alpha,\alpha\rangle/2=1-\rho_\nu\big\},
\end{equation*}
while $\Pi(g)=\emptyset$ if $d(\nu)>1$.
\end{proof}

Even if $d(\nu)=1$, it is possible for $\Pi(g)$ to be empty, for instance if the shortest vectors in $-h+\pi_{\nu}(\Lambda)$ have norm greater than $2(1-\rho_\nu)$. This is in particular the case if $\rho_\nu>1$.

\begin{prop}\label{prop:leechroots2}
The Dynkin diagram $\Phi(g)$ of $\Pi(g)$ can also be obtained as follows. Each vector $\alpha_i\in\Pi(g)\subseteq\h^*$ defines a node of $\Phi(g)$. The nodes $\alpha_i$ and $\alpha_j$ for $i\neq j$ are joined by
\begin{enumerate}
\item no edge if $\langle \alpha_i-\alpha_j,\alpha_i-\alpha_j\rangle/2 = 2(1-\rho_\nu)$,
\item a single edge if $\langle\alpha_i-\alpha_j,\alpha_i-\alpha_j\rangle/2 = 3(1-\rho_\nu)$,
\item an undirected double edge if $\langle\alpha_i-\alpha_j,\alpha_i-\alpha_j\rangle/2 = 4(1-\rho_\nu)$,
\end{enumerate}
corresponding to angles of $2\pi/4$, $2\pi/3$ and $2\pi/2$, respectively, between $\alpha_i$ and $\alpha_j$.
\end{prop}

We define the \emph{shifted weights}
\begin{equation*}
\tilde\Pi(g)\coloneqq\Pi(g)+h=\big\{\beta\in\pi_{\nu}(\Lambda)\,\big|\,\langle\beta-h,\beta-h\rangle/2=1-\rho_\nu\big\}\subseteq\pi_\nu(\Lambda).
\end{equation*}
We can associate a Dynkin diagram with $\tilde\Pi(g)$ in the same way as with $\Pi(g)$, using \autoref{prop:leechroots2}. Since the translation by $h$ does not affect the distances between the weights, both diagrams coincide. Geometrically, $\tilde\Pi(g)$ is given by the elements in $\pi_{\nu}(\Lambda)$ lying on the sphere in $\pi_{\nu}(\Lambda \otimes_\Z \R)$ with centre $h$ and radius $\sqrt{2(1-\rho_\nu)}$.

Examples of Dynkin diagrams inside the lattice $A_2$ (with different radii) are shown in \autoref{fig:diag}. The centres are coloured red, the diagrams blue.

\begin{figure}[ht]
\caption{Dynkin diagrams in the lattice $A_2$.}
~\\\includegraphics{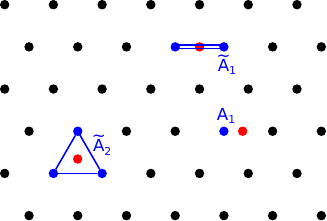}\\~
\label{fig:diag}
\end{figure}

In a connected extended affine Dynkin diagram with simple roots $\alpha_0,\ldots,\alpha_l$ there is a linear relation between the $\alpha_i$. More precisely, there are positive integers $a_i$ such that $\sum_{i=0}^l a_i\alpha_i=0$. If chosen coprime, the $a_i$ are unique and sometimes called \emph{Kac labels} (see, e.g., Table~Aff~1 in Section~4.8 of \cite{Kac90}).

\begin{prop}\label{prop:centre}
If $\Phi(g)$ contains a connected component of affine type, then the centre $h$ of $\tilde\Pi(g)$ can be reconstructed from the weights in $\tilde\Pi(g)$.
\end{prop}
\begin{proof}
Denote the shifted weights of the connected affine component by $\beta_0,\ldots,\beta_l$. Write $\beta_i=\alpha_i+h$. Then
\begin{equation*}
h=\sum_{i=0}^l a_i\beta_i\Big/\Big(\sum_{i=0}^la_i\Big).\qedhere
\end{equation*}
\end{proof}

We now additionally assume that the automorphism $g=\phi_\eta(\nu)\sigma_h\in\Aut(V_\Lambda)$ is extremal, i.e.\ that $g$ is a \gdh{}. Then $\rho(V_\Lambda(g))\geq 1$, so that
\begin{equation*}
\min_{\beta\in\pi_\nu(\Lambda)}\frac{\langle\beta-h,\beta-h\rangle}{2}\geq1-\rho_\nu
\end{equation*}
(see Proposition~5.9 in \cite{MS23}). Hence, if the hole diagram $\Phi(g)$ is non-empty, then the points in $\tilde\Pi(g)$ are exactly the \emph{closest vectors} to $h$ in $\pi_\nu(\Lambda)$.

As a side remark (cf.\ \cite{LM23}), we note that $h$ is in general not a \emph{deep hole} or even just a \emph{hole} of the lattice $\pi_\nu(\Lambda)$. Indeed, for most $\nu\in\O(\Lambda)$ the covering radius of $\pi_\nu(\Lambda)$ is greater than $\sqrt{2(1-\rho_\nu)}$ so that $h$ cannot be a deep hole of $\pi_\nu(\Lambda)$. In fact, usually the number of points in $\tilde\Pi(g)$ is less than $\rk(\pi_\nu(\Lambda))+1$, which means that $h$ cannot be a hole. On the other hand, if $\nu\in\O(\Lambda)$ is such that the covering radius of $\pi_\nu(\Lambda)$ is less than $\sqrt{2(1-\rho_\nu)}$, then there can be no extremal automorphism in $\Aut(V_{\Lambda})$ projecting to $\nu$.

We now exploit the fact that the inverse-orbifold automorphism of such a \gdh{} $g$ is known \cite{ELMS21}. Since we assumed that $g$ has order $n>1$, $\g=(V_{\Lambda}^{\orb(g)})_1$ must be semisimple, with decomposition $\g=\g_1\oplus\ldots\oplus\g_r$ into simple ideals. Recall that the inverse-orbifold automorphism is given by $\sigma_u=\e^{2\pi\i u_0}\in\Aut(V_{\Lambda}^{\orb(g)})$ with $u=\sum_{i=1}^r\rho_i/h_i^\vee$ where $h_i^\vee$ is the dual Coxeter number and $\rho_i$ the Weyl vector of $\g_i$ (see \autoref{sec:24}). The restriction of $\sigma_u$ to $\g$ only depends on the Lie algebra structure of $\g$, which means that the Dynkin diagram $\Phi(g)$ can be easily read off from the isomorphism type of $\g$:
\begin{prop}\label{prop:invtype}
Let $g$ be a \gdh{} of $V_\Lambda$ of order $n>1$ with $\rk((V_{\Lambda}^g)_1)>0$. Then $(V_{\Lambda}^{\orb(g)})_1=\g_1\oplus\ldots\oplus\g_r$ is semisimple and the Dynkin diagram $\Phi(g)$ is of type
\begin{equation*}
\bigcup_{\substack{i=1\\l_ih_i^\vee=n}}^r
\begin{cases}
\tilde{A}_l&\text{if }\g_i\text{ has type }A_l, l\geq1,\\
A_1&\text{if }\g_i\text{ has type }B_l, l\geq2,\\
A_{l-1}&\text{if }\g_i\text{ has type }C_l,l\geq3,\\
\tilde{D}_l&\text{if }\g_i\text{ has type }D_l,l\geq4,\\
\tilde{E}_l&\text{if }\g_i\text{ has type }E_l,l\in\{6,7,8\},\\
A_2&\text{if }\g_i\text{ has type }F_4,\\
A_1&\text{if }\g_i\text{ has type }G_2
\end{cases}
\end{equation*}
where $l_i\in\{1,2,3\}$ is the lacing number of the simple ideal $\g_i$.
\end{prop}
The order of $\sigma_u$ on each simple ideal $\g_i$ is $l_ih_i^\vee$ so that the order of $\sigma_u$ on $(V_{\Lambda}^{\orb(g)})_1$ is $\lcm(\{l_ih_i^\vee\}_{i=1}^r)$, which can be shown to equal the order $n$ of $\sigma_u$ on the whole \voa{} $V_{\Lambda}^{\orb(g)}$. The proposition states in particular that only those simple ideals contribute to the Dynkin diagram $\Phi(g)$, on which $\sigma_u$ assumes its order.
\begin{proof}
Recall that the inverse orbifold automorphism acts on $ (W_{\Lambda}^{(1,0)})_1 = V_{\Lambda}(g)_1$ as multiplication by $\e^{2\pi\i/n}$. Hence, the simple ideal $\g_i$ can only contribute to $(V_{\Lambda}(g))_1$ if the order of $\sigma_u$ restricted to $\g_i$, which is $l_ih_i^\vee$, equals $n$. On a simple ideal where this is the case, the eigenspace for the eigenvalue $\e^{2\pi\i/n}$ is now determined following Proposition~8.6~c) in \cite{Kac90}. For this one uses the type (in the language of \cite{Kac90}) of $\sigma_u$ restricted to $\g_i$, which is described in the proof of Proposition~5.1 in \cite{ELMS21}.
\end{proof}

The special case of the proposition for types $A$, $D$ and $E$ was already discussed in \cite{LS20} (see Lemma 2.6).

\medskip

From what we have seen so far, the Dynkin diagram $\Phi(g)$ of a \gdh{} could in principle be empty. The following is immediate:
\begin{cor}
Let $g$ be a \gdh{} of $V_\Lambda$ of order $n>1$ with $\rk((V_{\Lambda}^g)_1)>0$. Then the following are equivalent:
\begin{enumerate}
\item The Dynkin diagram $\Phi(g)$ is non-empty,
\item the set of shifted weights $\tilde\Pi(g)$ is non-empty,
\item $\rho(V_{\Lambda}(g))=1$,
\item $l_ih_i^\vee=\lcm(\{l_jh_j^\vee\}_{j=1}^r)$ for some $i\in\{1,\ldots,r\}$,
\item $|\sigma_u|=\big|\sigma_u|_{\g_i}\big|$ for some $i\in\{1,\ldots,r\}$.
\end{enumerate}
\end{cor}

\medskip

We now discuss the special case of $g$ being an inner automorphism. In this case, we exactly recover the classical hole diagrams in \cite{CPS82}:
\begin{prop}\label{prop:niemeiercase}
Let $g$ be a \gdh{} of $V_\Lambda$ of order $n>1$ with $\rk((V_{\Lambda}^g)_1)>0$. Assume that $g$ is inner. Then $g=\sigma_h$ for some deep hole $h\in {\Lambda \otimes_\Z\Q}$ corresponding to the Niemeier lattice $N$. Let $\tilde{N}$ be the extended affine Dynkin diagram corresponding to $N$, which is the hole diagram of $h$. Then $V_{\Lambda}^{\orb(g)} \cong V_N$ and $g$ has the generalised hole diagram $(1^{24}, \tilde{N})$.
\end{prop}
\begin{proof}
Since $g$ is inner, $g=\sigma_h$ for some $h\in\Lambda\otimes_\Z\Q$. The extremality of $g$ implies that $\rho(V(g))\geq1$. But the covering radius of the Leech lattice $\Lambda$ is $\sqrt{2}$, so that
\begin{equation*}
\rho(V(g))=\min_{\beta\in\Lambda}\frac{\langle\beta-h,\beta-h\rangle}{2}=1,
\end{equation*}
i.e.\ $h$ is a deep hole of $\Lambda$. The remaining claims follow from \autoref{prop:leechroots2} and the results in \cite{CPS82}.
\end{proof}


\section{Classification of \GDH{}s}\label{sec:gdhclass}

In this section we classify the \gdh{}s of the Leech lattice \voa{} by enumerating the corresponding generalised hole diagrams. As a consequence we obtain a new, geometric classification of the \strathol{} \voa{}s of central charge $24$ with non-trivial weight-$1$ space, which is independent of Schellekens' results.

\medskip

The possible generalised hole diagrams are strongly restricted by the following result (see Lemma 6.1 in \cite{ELMS21}):
\begin{prop}\label{prop:conditions}
Let $V$ be a \strathol{} \voa{} of central charge $24$ with $V_1$ semisimple. Let $\g_{1,k_1}\ldots\g_{r,k_r}$ denote the affine structure (with dual Coxeter numbers $h_i^\vee$ and lacing numbers $l_i$). Then
\begin{enumerate}
\item $h_i^\vee/k_i=(\dim(V_1)-24)/24$
\end{enumerate}
for all $i=1,\ldots,r$, and there exists a $\nu\in\O(\Lambda)$ such that
\begin{enumerate}[resume]\itemsep1.5mm
\item $\rk(\Lambda^\nu)=\rk(V_1)$,\label{item:cond2}
\item $|\nu|\,\big|\,\lcm(\{l_ih_i^\vee\}_{i=1}^r)$,\label{item:cond3}
\item $1/(1-\rho_\nu)=\lcm(\{l_ik_i\}_{i=1}^r)$.\label{item:cond4}
\end{enumerate}
\end{prop}
The automorphism $\nu$ is exactly the projection $\Aut(V_\Lambda)\to\O(\Lambda)$ of the \gdh{} $g$ corresponding to $V$ by \autoref{thm:bij}. Recall that $\rho_\nu$ denotes the vacuum anomaly of $\nu$ and only depends on the cycle shape of $\nu$.

The first equation is Schellekens' lowest-order trace identity \eqref{eq:221}. The other conditions follow from the bijection in \autoref{thm:bij}.

It is straightforward to list all solutions, i.e.\ pairs of affine structures and automorphisms of the Leech lattice $\Lambda$, to the equations in \autoref{prop:conditions} (see Proposition 6.2 in \cite{ELMS21}):
\begin{prop}\label{prop:13}
There are exactly $82$ pairs of affine structures and conjugacy classes in $\O(\Lambda)$ satisfying the four equations in \autoref{prop:conditions}. These are the $69$ cases described in \autoref{table:70} plus the $13$ spurious cases listed in \autoref{table:13}.
\end{prop}
\begin{table}[ht]
\caption{$13$ spurious cases in \autoref{prop:13}.}
\[
\begin{array}{lrcrlll}
\nu\in\O(\Lambda) & |\phi_\eta(\nu)|&  \rho_\nu & n & \text{Aff.\ Struct.} & \Phi(g) & \text{Norms}\\\hline\hline
6^4          & 12 & 35/36 &  6 & D_{4,36}                & \tilde{D}_4      & 2/18, 2/12 \\\hline
\multirow{2}{*}{$4^6$} & \multirow{2}{*}{$8$} & \multirow{2}{*}{$15/16$}
                          &  4 & A_{3,16}^2              & \tilde{A}_3^2    & 2/8, 6/16 \\
             &    &       &  8 & C_{3,8}A_{3,8}          & A_2              & 6/16 \\\hline
\multirow{3}{*}{$3^8$} & \multirow{3}{*}{$3$} & \multirow{3}{*}{$8/9$}
                          &  3 & A_{2,9}^4               & \tilde{A}_2^4    & 4/9, 2/3 \\
             &    &       &  6 & D_{4,9}A_{1,3}^4        & \tilde{D}_4      & 4/9, 2/3 \\
             &    &       & 12 & G_{2,3}^4               & A_1^4            & 4/9 \\\hline
\multirow{2}{*}{$2^44^4$} & \multirow{2}{*}{$4$} & \multirow{2}{*}{$7/8$}
                          &  4 & A_{3,8}^2A_{1,4}^2      & \tilde{A}_3^2    & 2/4, 6/8 \\
             &    &       &  8 & C_{3,4}^2A_{1,2}^2      & A_2^2            & 2/4, 6/8 \\\hline
1^22^23^26^2 &  6 & 5/6   &  6 & D_{4,6}B_{2,3}^2        & \tilde{D}_4A_1^2 & 2/3, 2/2 \\\hline
\multirow{4}{*}{$2^{12}$} & \multirow{4}{*}{$4$} & \multirow{4}{*}{$3/4$}
                          &  4 & A_{3,4}^2A_{1,2}^6      & \tilde{A}_3^2    & 2/2, 6/4 \\
             &    &       &  8 & D_{5,4}A_{3,2}A_{1,1}^4 & \tilde{D}_5      & 2/2, 6/4 \\
             &    &       &  8 & C_{3,2}^3A_{1,1}^3      & A_2^3            & 2/2, 6/4 \\
             &    &       &  8 & C_{3,2}^2A_{3,2}^2      & A_2^3            & 2/2, 6/4
\end{array}
\]
\label{table:13}
\end{table}
Note that there is no affine structure that appears in more than one pair. By \autoref{prop:invtype}, the affine structure fixes the generalised hole diagram of the corresponding \gdh{} $g$. However, there could still be multiple non-conjugate \gdh{}s for a given pair (or generalised hole diagram).

We observe that, except for $g=\id$, the Dynkin diagram $\Phi(g)$ of a \gdh{} is never empty.

\begin{lem}
There are no \gdh{}s in $\Aut(V_\Lambda)$ corresponding to the eight spurious cases in \autoref{table:13} with cycle shapes $6^4$, $4^6$, $3^8$ and $2^44^4$.
\end{lem}
\begin{proof}
We write the potential \gdh{} as $g=\phi_\eta(\nu)\sigma_h$ where $\phi_\eta(\nu)$ is a standard lift of $\nu\in\O(\Lambda)$. Note that $\langle\beta,\beta\rangle/2\in(1/|\phi_\eta(\nu)|)\Z$ for all $\beta\in\pi_\nu(\Lambda)$. The hole diagrams $\Phi(g)$ are determined by \autoref{prop:invtype} and listed in \autoref{table:13}. Based on \autoref{prop:leechroots2} we can also read off the norms of the differences of the elements in $\tilde\Pi(g)\subseteq\pi_\nu(\Lambda)$. Hence, none of the eight cases in the assertion can occur as in each case not all the computed norms are in $(2/|\phi_\eta(\nu)|)\Z$.
\end{proof}

As a consequence, we are left with $5$ spurious cases, namely those entries in \autoref{table:13} with cycle shapes $1^22^23^26^2$ and $2^{12}$.


\subsection{Affine Case}

Suppose that $g$ is a \gdh{} projecting to $\nu\in\O(\Lambda)$ and that the corresponding set of shifted weights $\tilde\Pi(g)\subseteq\pi_\nu(\Lambda)$ contains a connected affine component $\tilde{X}_l$. Our strategy will be to search for the Dynkin diagram $\tilde{X}_l$ inside $\pi_\nu(\Lambda)$ as lattice points lying on a sphere around some point $h\in\pi_\nu(\Lambda\otimes_\Z\Q)$ of radius $\sqrt{2(1-\rho_\nu)}$ with edges defined as in \autoref{prop:leechroots2} (see also \autoref{fig:diag}). We enumerate the occurrences of $\tilde{X}_l$ in $\pi_\nu(\Lambda)$, more precisely the finitely many orbits under the action of $C_{\O(\Lambda)}(\nu)\ltimes\pi_\nu(\Lambda)$. This can be done by moving one vertex of $\tilde{X}_l$ to the origin and then performing a short vector search in \texttt{Magma} \cite{Magma} (code available at \href{http://arxiv.org/abs/2112.12291}{\nolinkurl{arxiv.org}}). In principle, this could also be done by hand, as is demonstrated in \cite{CPS82} in the case $\nu=\id$. Note that $C_{\O(\Lambda)}(\nu)$ in general only induces a subgroup of $\O(\pi_\nu(\Lambda))$, but in view of \autoref{prop:leechconj} it is important to consider the orbits under $C_{\O(\Lambda)}(\nu)$ rather than the full orthogonal group $\O(\pi_\nu(\Lambda))$.

Then, since $\tilde{X}_l$ is of affine type, its centre $h$ is uniquely determined by the concrete realisation of $\tilde{X}_l$ inside $\pi_\nu(\Lambda)$ (see \autoref{prop:centre}). For each orbit, this immediately yields the complete hole diagram $\tilde{X}_l\ldots$ defined by $h$, which is some Dynkin diagram containing $\tilde{X}_l$ as connected component.

Finally, by \autoref{prop:leechconj}, each \gdh{} in $\Aut(V_\Lambda)$ defining a hole diagram containing $\tilde{X}_l$ in $\pi_\nu(\Lambda)$ must be conjugate to $g=\phi_\eta(\nu)\sigma_h$ where $\phi_\eta(\nu)$ is a standard lift of $\nu$ and $h$ is one of the centres in the finite list of orbits.

\medskip

Now, we go through the potential \gdh{}s in \autoref{table:70} and \autoref{table:13} containing a connected affine component ($54$ plus three spurious cases) and show that the entries of \autoref{table:13} cannot be realised by \gdh{}s while the candidates of \autoref{table:70} by at most one algebraic conjugacy class in $\Aut(V_{\Lambda})$.

We introduce the notation (cf.\ \cite{CPS82})
\begin{equation*}
\tilde{X}_l\implies\tilde{X}_l\ldots
\end{equation*}
to mean that there is a unique (unless otherwise noted) orbit under $C_{\O(\Lambda)}(\nu)\ltimes\pi_\nu(\Lambda)$ of the connected affine diagram $\tilde{X}_l$ in $\pi_\nu(\Lambda)$ (as lattice points sitting on a sphere of radius $\sqrt{2(1-\rho_\nu)}$ around the centre of $\tilde{X}_l$) and that it defines the complete diagram $\tilde{X}_l\ldots$ (\emph{all} the lattice points sitting on said sphere). If there are several orbits, each defining a different diagram $\tilde{X}_l\ldots$, we shall separate these by \emph{or}. If $\tilde{X}_l$ does not appear at all in $\pi_\nu(\Lambda)$, we write $\tilde{X}_l\implies\emptyset$.

\medskip

The first case was already covered in \cite{CPS82}:
\begin{lem}[\cite{CPS82}]\label{lem:deepholes}
Let $\nu\in\O(\Lambda)$ be of cycle shape $1^{24}$. Then in $\pi_\nu(\Lambda)=\Lambda$:
\begin{align*}
\tilde{A}_1&\implies\tilde{A}_1^{24},&                           \tilde{D}_4&\implies\tilde{D}_4^6\text{ or }\tilde{A}_5^4\tilde{D}_4,\\
\tilde{A}_2&\implies\tilde{A}_2^{12},&                           \tilde{D}_5&\implies\tilde{A}_7^2\tilde{D}_5^2,\\
\tilde{A}_3&\implies\tilde{A}_3^8,&                              \tilde{D}_6&\implies\tilde{D}_6^4\text{ or }\tilde{A}_9^2\tilde{D}_6,\\
\tilde{A}_4&\implies\tilde{A}_4^6,&                              \tilde{D}_7&\implies\tilde{A}_{11}\tilde{D}_7\tilde{E}_6,\\
\tilde{A}_5&\implies\tilde{A}_5^4\tilde{D}_4,&                   \tilde{D}_8&\implies\tilde{D}_8^3,\\
\tilde{A}_6&\implies\tilde{A}_6^4,&                              \tilde{D}_9&\implies\tilde{A}_{15}\tilde{D}_9,\\
\tilde{A}_7&\implies\tilde{A}_7^2\tilde{D}_5^2,&                 \tilde{D}_{10}&\implies\tilde{D}_{10}\tilde{E}_7^2,\\
\tilde{A}_8&\implies\tilde{A}_8^3,&                              \tilde{D}_{12}&\implies\tilde{D}_{12}^2,\\
\tilde{A}_9&\implies\tilde{A}_9^2\tilde{D}_6,&                   \tilde{D}_{16}&\implies\tilde{D}_{16}\tilde{E}_8,\\
\tilde{A}_{11}&\implies\tilde{A}_{11}\tilde{D}_7\tilde{E}_6,&    \tilde{D}_{24}&\implies\tilde{D}_{24},\\
\tilde{A}_{12}&\implies\tilde{A}_{12}^2,\\
\tilde{A}_{15}&\implies\tilde{A}_{15}\tilde{D}_9,&               \tilde{E}_6&\implies\tilde{E}_6^4\text{ or }\tilde{A}_{11}\tilde{D}_7\tilde{E}_6,\\
\tilde{A}_{17}&\implies\tilde{A}_{17}\tilde{E}_7,&               \tilde{E}_7&\implies\tilde{D}_{10}\tilde{E}_7^2\text{ or }\tilde{A}_{17}\tilde{E}_7,\\
\tilde{A}_{24}&\implies\tilde{A}_{24},&                          \tilde{E}_8&\implies\tilde{E}_8^3\text{ or }\tilde{D}_{16}\tilde{E}_8.
\end{align*}
In particular, there is at most one \gdh{} in $\Aut(V_\Lambda)$ up to conjugacy projecting to $\nu$ for each of the $23$ non-empty hole diagrams listed in \autoref{table:70}.
\end{lem}

\begin{lem}
Let $\nu\in\O(\Lambda)$ be of cycle shape $1^82^8$. Then in $\pi_\nu(\Lambda)$:
\begin{align*}
\tilde{A}_1&\implies\tilde{A}_1^8\text{ or }\tilde{A}_1^{16},\\
\tilde{A}_3&\implies A_1^4\tilde{A}_3^2\text{ or }\tilde{A}_3^4,\\
\tilde{A}_5&\implies A_1^2A_3\tilde{A}_5\text{ or }A_1\tilde{A}_5^2,\\
\tilde{A}_7&\implies A_1^2\tilde{A}_7\text{ (at most 2 orbits) or }A_2^2\tilde{A}_7,\\
\tilde{A}_9&\implies A_1\tilde{A}_9,\\
\tilde{D}_4&\implies A_1^8\tilde{D}_4\text{ or }\tilde{D}_4^2\text{ or }A_1^4\tilde{D}_4^2,\\
\tilde{D}_5&\implies D_4\tilde{D}_5\text{ or }\tilde{D}_5^2,\\
\tilde{D}_6&\implies A_1^4\tilde{D}_6\text{ or }A_1^2A_3\tilde{D}_6,\\
\tilde{D}_8&\implies\tilde{D}_8\text{ (at most 2 orbits) or }A_1^2\tilde{D}_8,\\
\tilde{D}_9&\implies\tilde{D}_9,\\
\tilde{E}_6&\implies A_3\tilde{E}_6\text{ or }A_4\tilde{E}_6,\\
\tilde{E}_7&\implies A_1^2\tilde{E}_7\text{ or }A_1A_2\tilde{E}_7,\\
\tilde{E}_8&\implies\tilde{E}_8\text{ or }A_1\tilde{E}_8.
\end{align*}
In particular, there is at most one \gdh{} in $\Aut(V_\Lambda)$ up to conjugacy projecting to $\nu$ for each of the hole diagrams $A_1\tilde{E}_8$, $A_1A_2\tilde{E}_7$, $\tilde{D}_9$, $A_1^2\tilde{D}_8$, $A_4\tilde{E}_6$, $A_1\tilde{A}_9$, $A_1^2A_3\tilde{D}_6$, $A_2^2\tilde{A}_7$, $\tilde{D}_5^2$, $A_1\tilde{A}_5^2$, $A_1^4\tilde{D}_4^2$, $\tilde{A}_3^4$ and $\tilde{A}_1^{16}$.
\end{lem}
We can explicitly check, for instance, that the automorphism $g=\phi_\eta(\nu)\sigma_h$ defined by the diagram $\tilde{A}_1^{16}$ and its centre $h$ is a \gdh{}, while for the diagram $\tilde{A}_1^8$ this is not the case.

\begin{lem}
Let $\nu\in\O(\Lambda)$ be of cycle shape $1^63^6$. Then in $\pi_\nu(\Lambda)$:
\begin{align*}
\tilde{A}_2&\implies\tilde{A}_2^3\text{ or }\tilde{A}_2^6,\\
\tilde{A}_5&\implies A_2\tilde{A}_5\text{ (at most 2 orbits) or }\tilde{A}_5\tilde{D}_4,\\
\tilde{A}_8&\implies\tilde{A}_8,\\
\tilde{D}_4&\implies A_2^2\tilde{D}_4\text{ or }\tilde{A}_5\tilde{D}_4,\\
\tilde{D}_7&\implies A_1\tilde{D}_7,\\
\tilde{E}_6&\implies\tilde{E}_6\text{ (at most 2 orbits) or }A_1^3\tilde{E}_6,\\
\tilde{E}_7&\implies\tilde{E}_7.
\end{align*}
In particular, there is at most one \gdh{} in $\Aut(V_\Lambda)$ up to conjugacy projecting to $\nu$ for each of the hole diagrams $\tilde{E}_7$, $A_1\tilde{D}_7$, $A_1^3\tilde{E}_6$, $\tilde{A}_8$, $\tilde{A}_5\tilde{D}_4$ and $\tilde{A}_2^6$.
\end{lem}

For the cycle shape $2^{12}$ we first remove two more spurious cases.

\begin{lem}
Let $\nu\in\O(\Lambda)$ be of cycle shape $2^{12}$. Then in $\pi_\nu(\Lambda)$:
\begin{align*}
\tilde{A}_3&\implies\tilde{A}_3,\\
\tilde{D}_5&\implies\emptyset.
\end{align*}
In particular, there is no \gdh{} in $\Aut(V_\Lambda)$ projecting to $\nu$ with hole diagram $\tilde{A}_3^2$ or $\tilde{D}_5$.
\end{lem}

\begin{lem}
Let $\nu\in\O(\Lambda)$ be of cycle shape $2^{12}$. Then in $\pi_\nu(\Lambda)$:
\begin{align*}
\tilde{A}_1&\implies\tilde{A}_1^4\text{ or }\tilde{A}_1^{12},\\
\tilde{D}_4&\implies\tilde{D}_4\text{ (2 orbits)}.
\end{align*}
One of the two orbits of type $\tilde{D}_4$ has a centre defining an automorphism of order $12$, the other one an automorphism of order $6$.

In particular, there is at most one \gdh{} in $\Aut(V_\Lambda)$ up to conjugacy projecting to $\nu$ for each of the hole diagrams $\tilde{A}_1^{12}$ and $\tilde{D}_4$.
\end{lem}

\begin{lem}
Let $\nu\in\O(\Lambda)$ be of cycle shape $1^42^24^4$. Then in $\pi_\nu(\Lambda)$:
\begin{align*}
\tilde{A}_3&\implies\tilde{A}_3^2\text{ or }A_1^2\tilde{A}_3\text{ (at most 2 orbits) or }A_1^4\tilde{A}_3\text{ or }\tilde{A}_3^3,\\
\tilde{A}_7&\implies\tilde{A}_7,\\
\tilde{D}_5&\implies\tilde{D}_5\text{ or }A_2\tilde{D}_5,\\
\tilde{E}_6&\implies\tilde{E}_6.
\end{align*}
In particular, there is at most one \gdh{} in $\Aut(V_\Lambda)$ up to conjugacy projecting to $\nu$ for each of the hole diagrams $\tilde{E}_6$, $\tilde{A}_7$, $A_2\tilde{D}_5$ and $\tilde{A}_3^3$.
\end{lem}

\begin{lem}
Let $\nu\in\O(\Lambda)$ be of cycle shape $1^45^4$. Then in $\pi_\nu(\Lambda)$:
\begin{align*}
\tilde{A}_4&\implies\tilde{A}_4\text{ (at most 3 orbits) or }\tilde{A}_4^2,\\
\tilde{D}_6&\implies\tilde{D}_6.
\end{align*}
In particular, there is at most one \gdh{} in $\Aut(V_\Lambda)$ up to conjugacy projecting to $\nu$ for each of the hole diagrams $\tilde{D}_6$ and $\tilde{A}_4^2$.
\end{lem}

We remove the spurious case for the cycle shape $1^22^23^26^2$.

\begin{lem}
Let $\nu\in\O(\Lambda)$ be of cycle shape $1^22^23^26^2$. Then in $\pi_\nu(\Lambda)$:
\begin{equation*}
\tilde{D}_4\implies\emptyset
\end{equation*}
In particular, there is no \gdh{} in $\Aut(V_\Lambda)$ projecting to $\nu$ with hole diagram $A_1^2\tilde{D}_4$.
\end{lem}

\begin{lem}
Let $\nu\in\O(\Lambda)$ be of cycle shape $1^22^23^26^2$. Then in $\pi_\nu(\Lambda)$:
\begin{equation*}
\tilde{A}_5\implies A_1\tilde{A}_5.
\end{equation*}
In particular, there is at most one \gdh{} in $\Aut(V_\Lambda)$ up to conjugacy projecting to $\nu$ with hole diagram $A_1\tilde{A}_5$.
\end{lem}

\begin{lem}
Let $\nu\in\O(\Lambda)$ be of cycle shape $1^37^3$. Then in $\pi_\nu(\Lambda)$:
\begin{equation*}
\tilde{A}_6\implies\tilde{A}_6.
\end{equation*}
In particular, there is at most one \gdh{} in $\Aut(V_\Lambda)$ up to conjugacy projecting to $\nu$ with hole diagram $\tilde{A}_6$.
\end{lem}

\begin{lem}
Let $\nu\in\O(\Lambda)$ be of cycle shape $1^22^14^18^2$. Then in $\pi_\nu(\Lambda)$:
\begin{equation*}
\tilde{D}_5\implies\tilde{D}_5.
\end{equation*}
In particular, there is at most one \gdh{} in $\Aut(V_\Lambda)$ up to conjugacy projecting to $\nu$ with hole diagram $\tilde{D}_5$.
\end{lem}

\begin{lem}
Let $\nu\in\O(\Lambda)$ be of cycle shape $2^36^3$. Then in $\pi_\nu(\Lambda)$:
\begin{equation*}
\tilde{D}_4\implies\tilde{D}_4.
\end{equation*}
In particular, there is at most one \gdh{} in $\Aut(V_\Lambda)$ up to conjugacy projecting to $\nu$ with hole diagram $\tilde{D}_4$.
\end{lem}


\subsection{Non-Affine Case}

We now consider the more difficult case of potential \gdh{}s $g$ with hole diagrams that do not contain any affine component. These are $15$ plus two spurious cases (see \autoref{table:70} and \autoref{table:13}).

First, we enumerate the orbits under $C_{\O(\Lambda)}(\nu)\ltimes\pi_\nu(\Lambda)$ of the diagram realised in $\pi_\nu(\Lambda)$ as lattice points with relative distances defined as in \autoref{prop:leechroots2}. This is the same computation as in the affine case, with the exception that we are now directly searching for the complete diagram. Again, this is a relatively cheap computation using the short vector search in \texttt{Magma} \cite{Magma} (code made available at \href{http://arxiv.org/abs/2112.12291}{\nolinkurl{arxiv.org}}).

The points of the hole diagram must lie on a sphere of radius $\sqrt{2(1-\rho_\nu)}$ around some centre $h \in\pi_\nu(\Lambda\otimes_\Z\Q)$. However, in contrast to the affine case, this centre is not uniquely determined by the diagram. (In the most extreme case of the diagram $A_1$, $h$ could be any point at distance $\sqrt{2(1-\rho_\nu)}$ from the single vertex defining $A_1$.) The second and generally computationally more expensive part is to determine all the possible $h\in\pi_\nu(\Lambda\otimes_\Z\Q)$ that could be the centre of the diagram.

We employ two methods to facilitate this search. First, from \autoref{prop:13} we know the order $n$ of the \gdh{} (see \autoref{table:70} and \autoref{table:13}). Then \autoref{prop:leechconj2} implies that $h$ must lie in $s_\nu+\Lambda^\nu/n$ (where $s_\nu$ is non-zero if and only if $\nu$ exhibits order doubling). Second, as $h$ must have distance $\sqrt{2(1-\rho_\nu)}$ to all the vertices in the hole diagram, it must in particular lie on the hyperplanes of points equidistant to all pairs of vertices. This reduces the dimension of the eventual close vector search to find $h$, which is again performed in \texttt{Magma} \cite{Magma}.

As a result, for each orbit of the original diagram search we obtain a finite list of possible centres $h$. We then only keep those $h$
\begin{enumerate}
\item\label{item:cent1} whose corresponding automorphism $g=\phi_\eta(\nu)\sigma_h$ with standard lift $\phi_\eta(\nu)$ has order $n$ (i.e.\ $g$ must satisfy $\lcm(|\nu|,|\sigma_{h-s_\nu}|)=n$),
\item\label{item:cent2} such that $\tilde\Pi(g)=\big\{\beta\in\pi_{\nu}(\Lambda)\,\big|\,\langle\beta-h,\beta-h\rangle/2=1-\rho_\nu\big\}$ has exactly the diagram we are searching for (a priori we only know that it contains this diagram as a subdiagram),
\item\label{item:cent3} that actually correspond to a \gdh{} $g=\phi_\eta(\nu)\sigma_h$ (in particular, $g$ must be extremal).
\end{enumerate}

\medskip

Again, we sort the results by cycle shape and treat the case $2^{12}$ last because it is the most complicated one.

\medskip

\begin{lem}
Let $\nu\in\O(\Lambda)$ be of cycle shape $1^82^8$. Then in $\pi_\nu(\Lambda)$ there is exactly:

\smallskip

One orbit under $C_{\O(\Lambda)}(\nu)\ltimes\pi_\nu(\Lambda)$ of the diagram $A_1A_9$. There are $16$ possible centres $h\in\Lambda^\nu/22$ of this diagram, but only one of them satisfies \eqref{item:cent1} to \eqref{item:cent3}.

One orbit under $C_{\O(\Lambda)}(\nu)\ltimes\pi_\nu(\Lambda)$ of the diagram $A_2^2A_7$. There are six possible centres $h\in\Lambda^\nu/18$ of this diagram, but only two of them satisfy \eqref{item:cent1} to \eqref{item:cent3}. They are both in the same orbit under $C_{\O(\Lambda)}(\nu)\ltimes\pi_\nu(\Lambda)$.

One orbit under $C_{\O(\Lambda)}(\nu)\ltimes\pi_\nu(\Lambda)$ of the diagram $A_1A_5^2$. There are seven possible centres $h\in\Lambda^\nu/14$ of this diagram, but only one of them satisfies \eqref{item:cent1} to \eqref{item:cent3}.

Two orbits under $C_{\O(\Lambda)}(\nu)\ltimes\pi_\nu(\Lambda)$ of the diagram $A_3^4$. For the first orbit there is one possible centre $h\in\Lambda^\nu/10$ and it satisfies \eqref{item:cent1} to \eqref{item:cent3}. For the second orbit there are six possible centres $h\in\Lambda^\nu/10$, but none of them satisfy \eqref{item:cent1} to \eqref{item:cent3}.

\smallskip

In particular, there is at most one \gdh{} in $\Aut(V_\Lambda)$ up to conjugacy projecting to $\nu$ for each of the hole diagrams $A_1A_9$, $A_2^2A_7$, $A_1A_5^2$ and $A_3^4$.
\end{lem}

\begin{lem}
Let $\nu\in\O(\Lambda)$ be of cycle shape $1^42^24^4$. Then in $\pi_\nu(\Lambda)$ there is exactly:

\smallskip

One orbit under $C_{\O(\Lambda)}(\nu)\ltimes\pi_\nu(\Lambda)$ of the diagram $A_6$. There are nine possible centres $h\in\Lambda^\nu/16$ of this diagram, but only one of them satisfies \eqref{item:cent1} to \eqref{item:cent3}.

\smallskip

In particular, there is at most one \gdh{} in $\Aut(V_\Lambda)$ up to conjugacy projecting to $\nu$ with hole diagram $A_6$.
\end{lem}

\begin{lem}
Let $\nu\in\O(\Lambda)$ be of cycle shape $1^22^23^26^2$. Then in $\pi_\nu(\Lambda)$ there is exactly:

\smallskip

One orbit under $C_{\O(\Lambda)}(\nu)\ltimes\pi_\nu(\Lambda)$ of the diagram $A_1A_4$. There are seven possible centres $h\in\Lambda^\nu/12$ of this diagram, but only two of them satisfy \eqref{item:cent1} to \eqref{item:cent3}. They are both in the same orbit under $C_{\O(\Lambda)}(\nu)\ltimes\pi_\nu(\Lambda)$.

\smallskip

In particular, there is at most one \gdh{} in $\Aut(V_\Lambda)$ up to conjugacy projecting to $\nu$ with hole diagram $A_1A_4$.
\end{lem}

\begin{lem}
Let $\nu\in\O(\Lambda)$ be of cycle shape $2^36^3$. Then in $\pi_\nu(\Lambda)$ there is exactly:

\smallskip

One orbit under $C_{\O(\Lambda)}(\nu)\ltimes\pi_\nu(\Lambda)$ of the diagram $A_2$. There are $98$ possible centres $h\in s_\nu+\Lambda^\nu/18$ of this diagram, but only $36$ of them satisfy \eqref{item:cent1} to \eqref{item:cent3}. They are all in the same orbit under $C_{\O(\Lambda)}(\nu)\ltimes\pi_\nu(\Lambda)$.

\smallskip

In particular, there is at most one \gdh{} in $\Aut(V_\Lambda)$ up to conjugacy projecting to $\nu$ with hole diagram $A_2$.
\end{lem}

\begin{lem}
Let $\nu\in\O(\Lambda)$ be of cycle shape $2^210^2$. Then in $\pi_\nu(\Lambda)$ there is exactly:

\smallskip

One orbit under $C_{\O(\Lambda)}(\nu)\ltimes\pi_\nu(\Lambda)$ of the diagram $A_3$. There are two possible centres $h\in s_\nu+\Lambda^\nu/10$ of this diagram and both of them satisfy \eqref{item:cent1} to \eqref{item:cent3}. They are both in the same orbit under $C_{\O(\Lambda)}(\nu)\ltimes\pi_\nu(\Lambda)$.

\smallskip

In particular, there is at most one \gdh{} in $\Aut(V_\Lambda)$ up to conjugacy projecting to $\nu$ with hole diagram $A_3$.
\end{lem}

Now we consider the case $2^{12}$. We start by excluding the last two spurious cases.

\begin{lem}
Let $\nu\in\O(\Lambda)$ be of cycle shape $2^{12}$. There is no hole diagram in $\pi_\nu(\Lambda)$ containing $A_2^3$.

In particular, there is no \gdh{} in $\Aut(V_\Lambda)$ projecting to $\nu$ with hole diagram $A_2^3$.
\end{lem}

\begin{lem}
Let $\nu\in\O(\Lambda)$ be of cycle shape $2^{12}$. Then in $\pi_\nu(\Lambda)$ there is exactly:

\smallskip

One orbit under $C_{\O(\Lambda)}(\nu)\ltimes\pi_\nu(\Lambda)$ of the diagram $A_2$. There are $9{,}132{,}200$ possible centres $h\in s_\nu+\Lambda^\nu/18$ of this diagram, but only $31{,}680$ of them satisfy \eqref{item:cent1} to \eqref{item:cent3}. They are all in the same orbit under $C_{\O(\Lambda)}(\nu)\ltimes\pi_\nu(\Lambda)$.

One orbit under $C_{\O(\Lambda)}(\nu)\ltimes\pi_\nu(\Lambda)$ of the diagram $A_3$. There are $432$ possible centres $h\in\Lambda^\nu/10$ of this diagram, but only $72$ of them satisfy \eqref{item:cent1} to \eqref{item:cent3}. They are all in the same orbit under $C_{\O(\Lambda)}(\nu)\ltimes\pi_\nu(\Lambda)$.

\smallskip

In particular, there is at most one \gdh{} in $\Aut(V_\Lambda)$ up to conjugacy projecting to $\nu$ for each of the hole diagrams $A_2$ and $A_3$.
\end{lem}

We now discuss the most difficult cases, the \gdh{}s with cycle shape $2^{12}$ and hole diagrams $A_1$, $A_1^2$, $A_1^3$, $A_1^4$ and $A_1^6$. The hardest case is the potential \gdh{} of order $46$ with hole diagram $A_1$. Here, the hole diagram merely implies that the vector $h$ is the only and closest point at distance $1-\rho_\nu=1/4$ from the single vertex $\beta_1\in\pi_\nu(\Lambda)$ defining $A_1$.

Fortunately, we can exploit that the fixed-point lattice $\Lambda^\nu$ has a very symmetric embedding into Euclidean space. Indeed, let $D_{12}^+$ denote the positive-definite, integral lattice
\begin{equation*}
D_{12}^+\coloneqq\Big\{(x_1,\ldots,x_{12})\in\R^{12}\,\Big|\,\text{all }x_i\in\Z\text{ or all }x_i\in\Z+1/2 \text{ and }\sum_{i=1}^{12}x_i\in 2\Z \Big\}
\end{equation*}
embedded into Euclidean space $\R^{12}$ with the standard scalar product. It is the unique indecomposable, positive-definite, integral, unimodular lattice of rank $12$. Let $K\coloneqq\sqrt{2}D_{12}^+$ denote the lattice with lattice vectors scaled by $\sqrt{2}$. Then $K$ is even and
\begin{equation*}
\Lambda^\nu\cong K.
\end{equation*}
Moreover, we note that
\begin{equation*}
\pi_\nu(\Lambda)=(\Lambda^\nu)'=\Lambda^\nu/2.
\end{equation*}
The first equality holds since $\Lambda$ is unimodular, but the second equality (which is not just an isomorphism but a proper equality) is a special property of $\nu$.

The automorphism group of $D_{12}^+$ (and of $K$) is generated by permutations and even sign changes, i.e.\
\begin{equation*}
\O(K)=S_{12}\ltimes2^{11}.
\end{equation*}
The kernel of the map $C_{\O(\Lambda)}(\nu)\to\O(\Lambda^\nu)\cong\O(K)$ has order $2$ and is generated by $\nu$. The image has index $5040$ and is of the form $P\ltimes2^{11}$ where $P$ is some permutation group of index $5040$ in $S_{12}$.

In the following, we want to show that there is a unique $h\in\pi_\nu(\Lambda\otimes_\Z\Q)/\pi_\nu(\Lambda)\cong (K\otimes_\Z\Q)/(K/2)$ with a certain list of properties up to the action of $C_{\O(\Lambda)}(\nu)$. The number of elements we have to consider is too big to be amenable to a brute-force approach. We therefore split the computation into three parts, first considering only properties invariant under the much bigger group $S_{12}\ltimes2^{12}$ (where we allow all sign changes) and computing the orbits satisfying these, then under the group $\O(K)=S_{12}\ltimes2^{11}$ and finally under the group $C_{\O(\Lambda)}(\nu)$, i.e.\ $P\ltimes2^{11}$. In each step the number of orbits we consider remains manageable.

\begin{lem}
Let $\nu\in\O(\Lambda)$ be of cycle shape $2^{12}$. Then there is at most one \gdh{} in $\Aut(V_\Lambda)$ up to conjugacy projecting to $\nu$ for each of the hole diagrams $A_1$, $A_1^2$, $A_1^3$, $A_1^4$ and $A_1^6$.
\end{lem}
\begin{proof}
We only describe the hardest case of the diagram $A_1$. The other cases are treated analogously.

A \gdh{} with hole diagram $A_1$ has order $46$ and is conjugate to $g=\phi_\eta(\nu)\sigma_h$ for some standard lift $\phi_\eta(\nu)$ and $h\in s_\nu+\Lambda^\nu/46\subseteq\Lambda^\nu/92$. By applying a translation in $\pi_\nu(\Lambda)$, we may assume that the only vertex of the hole diagram $A_1$ in $\pi_\nu(\Lambda)$ is the origin. Then $h$ has the following properties:
\begin{enumerate}
\item $h\in\Lambda^\nu/92$,
\item $\langle h,h\rangle/2=1/4$,
\item $\langle h-\beta,h-\beta\rangle/2\geq 1/4$ for all $\beta\in\pi_\nu(\Lambda)$,
\item $\langle h-\beta,h-\beta\rangle/2=1/4$ and $\beta\in\pi_\nu(\Lambda)$ if and only if $\beta=0$.
\end{enumerate}
We consider $\tilde{h}\coloneqq92h$. The above conditions are equivalent to:
\begin{enumerate}
\item[(1)] $\tilde{h}\in\Lambda^\nu$,
\item[(2)] $\langle\tilde{h},\tilde{h}\rangle/2=46^2$,
\item[(3)] $\langle\tilde{h},\beta\rangle\leq46\langle\beta,\beta\rangle/2$ for all $\beta\in\Lambda^\nu$,
\item[(4)] $\langle\tilde{h},\beta\rangle=46\langle\beta,\beta\rangle/2$ and $\beta\in\Lambda^\nu$ if and only if $\beta=0$.
\end{enumerate}
We identify $\Lambda^\nu$ with $K$ and write $\tilde{h}=\sqrt{2}(h_1,\ldots,h_{12})$. Then the first condition is equivalent to  either all $h_i\in\Z$ or all $h_i\in\Z+1/2$, and moreover $\sum_{i=1}^{12}h_i\in 2\Z$. We actually know that $\tilde{h}\in 92s_\nu+2\Lambda^\nu$ for some $s_\nu\in\Lambda^\nu/4$ (such that $\phi_\eta(\nu)\sigma_{s_\nu}$ has order~$2$). By choosing an $s_\nu$ we see that either all $h_i\in2\Z$ or all $h_i\in2\Z+1$. In total, the above conditions imply:
\begin{enumerate}
\item[(1')] all $h_i\in2\Z$ or all $h_i\in2\Z+1$,
\item[(2')] $\sum_{i=1}^{12}h_i^2=46^2$,
\item[(3')] $|h_i|+|h_j|<46$ for $i\neq j$.
\end{enumerate}
We determine the orbits of the solutions of these three conditions up to the action of $S_{12} \ltimes 2^{12}$, i.e.\ we ignore signs and permutations. This is a simple combinatorial problem with $10,301$ solutions.

\smallskip

We then consider the corresponding orbits under $\O(K)=S_{12}\ltimes2^{11}$, i.e.\ each orbit represented by a sequence $(h_1,\ldots,h_{12})$ not containing a $0$ splits up into two orbits by introducing a sign at, e.g., the first entry. The fact that $g$ is extremal implies that the twisted modules $V(g)$, $V(g^5)$, $V(g^9)$, $V(g^{13})$, $V(g^{17})$ and $V(g^{21})$ each have conformal weight at least~$1$. Since $\phi_\eta(\nu)^4=\id$, it follows that $g^{4k+1}=\phi_\eta(\nu)\sigma_{(4k+1)h}$ so that these conditions translate to
\begin{equation}\tag{4'}
\min_{\beta\in 46\Lambda^\nu}\frac{\langle(4k+1)\tilde{h}-\beta,(4k+1)\tilde{h}-\beta\rangle}{2}=46^2
\end{equation}
for $k=0,1,\ldots,11$. In fact, for $k=0$ we require equality and that there is exactly one closest vector, namely the one forming the diagram $A_1$. These conditions are invariant under $\O(K)=S_{12}\ltimes2^{11}$. The result is that there is exactly one orbit under $\O(K)$ satisfying conditions (1') to (4'), namely $(0,2,4,6,8,10,12,14,16,18,20,24)$.

\smallskip

Finally, we split up this orbit into the orbits under the action of the centraliser $C_{\O(\Lambda)}(\nu)$, i.e.\ under $P\ltimes2^{11}$. In this case, since all the $h_i$ are distinct, these orbits are in natural bijection with the $5040$ cosets of $P$ in $S_{12}$, which can be computed using \texttt{GAP} \cite{GAP4.10.2}. For these orbits we then explicitly check if they can be \gdh{}s of order $46$, i.e.\ in particular extremal. In the end, this leaves us with just one orbit $h \in \pi_\nu(\Lambda\otimes_\Z\Q)/\pi_\nu(\Lambda)$ under the action of $C_{\O(\Lambda)}(\nu)$, which concludes the proof.
\end{proof}

We remark that the \gdh{}s (of order $n$) for the diagrams $A_1$, $A_1^2$, $A_1^3$, $A_1^4$ and $A_1^6$ correspond to the vectors $h=\sqrt{2}(h_1,\ldots,h_{12})/(2n)$ in $K/(2n)\subseteq\R^{12}$ specified by the following $h_i$:
\[
\begin{array}{l|rrrrrrrrrrrr}
A_1              &  0 & 2 & 4 & 6 & 8 & 10 & 12 & 14 & 16 & 18 & 20 & 24\\
A_1^2            &  0 & 0 & 2 & 2 & 4 &  4 &  6 &  6 &  8 &  8 & 10 & 12\\
A_1^3            &  0 & 0 & 0 & 2 & 2 &  2 &  4 &  4 &  4 &  6 &  6 &  8\\
A_1^4            &  0 & 0 & 0 & 0 & 2 &  2 &  2 &  2 &  4 &  4 &  4 &  6\\
A_1^6            &  0 & 0 & 0 & 0 & 0 &  0 &  2 &  2 &  2 &  2 &  2 &  4\\
\tilde{A}_1^{12} &  0 & 0 & 0 & 0 & 0 &  0 &  0 &  0 &  0 &  0 &  0 &  2
\end{array}
\]
Here, we ignore signs and the order of the entries, which in any case depend on the concrete choice of the isomorphism $\Lambda^\nu\cong K$.


\subsection{Classification Results}

We summarise the above results in:
\begin{prop}\label{prop:coolesache}
There are at most $70$ conjugacy classes of \gdh{}s $g$ in $\Aut(V_\Lambda)$ with $\rk((V_{\Lambda}^g)_1)>0$. They are described in \autoref{table:70}.
\end{prop}

In \cite{MS23} we list $70$ \gdh{}s $g$ in $\Aut(V_\Lambda)$ with $\rk((V_{\Lambda}^g)_1)>0$. Using \autoref{prop:invtype} we can easily determine their generalised hole diagrams, which are all distinct. This implies the main result:
\begin{thm}[Classification of \GDH{}s]\label{thm:gdhclass}
There are exactly $70$ conjugacy classes of \gdh{}s $g$ in $\Aut(V_\Lambda)$ with $\rk((V_{\Lambda}^g)_1)>0$. The conjugacy class of $g$ is uniquely fixed by its generalised hole diagram.
\end{thm}

An automorphism $g$ of order $n$ is called \emph{rational} if $g$ is conjugate to $g^i$ for all $i\in\Z_n$ with $(i,n)=1$ (see, e.g., Chapter~7 in \cite{Ser08}). Equivalently, the conjugacy class and the algebraic conjugacy class (i.e.\ the conjugacy class of the cyclic subgroup) of $g$ coincide. The following observation is immediate:
\begin{cor}
The \gdh{}s $g$ in $\Aut(V_\Lambda)$ with $\rk((V_{\Lambda}^g)_1)>0$ are rational, i.e.\ conjugacy is equivalent to algebraic conjugacy.
\end{cor}
We also recover the decomposition of the Schellekens \voa{}s into $12$ families described by Höhn in \cite{Hoe17} (cf.\ \cite{Mor21,Mor23}):
\begin{thm}[Projection to $\Co_0$]\label{thm:proj11}
Under the natural projection $\Aut(V_\Lambda)\to\O(\Lambda)$ the \gdh{}s $g$ of $V_\Lambda$ with $\rk((V_\Lambda^g)_1)>0$ map to the $11$ algebraic conjugacy classes in $\O(\Lambda)\cong\Co_0$ with cycle shapes $1^{24}$, $1^82^8$, $1^63^6$, $2^{12}$, $1^42^24^4$, $1^45^4$, $1^22^23^26^2$, $1^37^3$, $1^22^14^18^2$, $2^36^3$ and $2^210^2$.
\end{thm}

A consequence of the above classification of \gdh{}s and the holey correspondence in \cite{MS23} is a new, geometric proof of the following result:
\begin{thm}[Classification of \VOA{}s]\label{thm:voaclass}
Up to isomorphism there are exactly $70$ \strathol{} \voa{}s $V$ of central charge $24$ with $V_1\neq\{0\}$. Such a \voa{} is uniquely determined by its $V_1$-structure.
\end{thm}

We have thus obtained a geometric proof of this classification that is analogous to the classification of the Niemeier lattices by enumeration of the corresponding deep holes of the Leech lattice $\Lambda$ \cite{CPS82,Bor85b}. In fact, it includes it as a special case (see \autoref{prop:niemeiercase}).

\medskip

We mention that \cite{LM23} give an interpretation of the \gdh{}s of $V_\Lambda$ in terms of actual deep holes of $\Lambda$ after rescaling.

Moreover, we remark that Höhn's approach to the classification problem in \cite{Hoe17} (and \cite{Lam20}) based on coset constructions can also be used to give a uniform proof of the above classification result \cite{Hoe17,BLS23}.

\begin{table}[p]
\caption{The $70$ \gdh{}s of $V_\Lambda$ whose corresponding orbifold constructions realise all non-zero Lie algebras on Schellekens' list (continued on next page).}
\renewcommand{\arraystretch}{1.09}
\begin{tabular}{rllrrll}
No. & No. & $(V_\Lambda^{\orb(g)})_1$ & Dim. & $n$ & $\rho(V_\Lambda(g^m))$ & $\Phi(g)$\\\hline\hline
\multicolumn{6}{c}{Rk. 24, cycle shape $1^{24}$}\\\hline\\[-1.2em]
70 &  A1 & $D_{24,1}$                  & 1128 & 46 & $1, 22/23, 1, 0$                  & $\tilde{D}_{24}$\\
69 &  A2 & $D_{16,1}E_{8,1}$           &  744 & 30 & $1, 14/15, 1, 1, 1, 1, 1, 0$      & $\tilde{D}_{16}\tilde{E}_8$\\
68 &  A3 & $E_{8,1}^3$                 &  744 & 30 & $1, 14/15, 9/10, 5/6, 1, 1, 1, 0$ & $\tilde{E}_8^3$\\
67 &  A4 & $A_{24,1}$                  &  624 & 25 & $1, 1, 0$                         & $\tilde{A}_{24}$\\
66 &  A5 & $D_{12,1}^2$                &  552 & 22 & $1, 10/11, 1, 0$                  & $\tilde{D}_{12}^2$\\
65 &  A6 & $A_{17,1}E_{7,1}$           &  456 & 18 & $1, 1, 1, 1, 1, 0$                & $\tilde{A}_{17}\tilde{E}_7$\\
64 &  A7 & $D_{10,1}E_{7,1}^2$         &  456 & 18 & $1, 8/9, 1, 1, 1, 0$              & $\tilde{D}_{10}\tilde{E}_7^2$\\
63 &  A8 & $A_{15,1}D_{9,1}$           &  408 & 16 & $1, 1, 1, 1, 0$                   & $\tilde{A}_{15}\tilde{D}_9$\\
61 &  A9 & $D_{8,1}^3$                 &  360 & 14 & $1, 6/7, 1, 0$                    & $\tilde{D}_8^3$\\
60 & A10 & $A_{12,1}^2$                &  336 & 13 & $1, 0$                            & $\tilde{A}_{12}^2$\\
59 & A11 & $A_{11,1}D_{7,1}E_{6,1}$    &  312 & 12 & $1, 1, 1, 1, 1, 0$                & $\tilde{A}_{11}\tilde{D}_7\tilde{E}_6$\\
58 & A12 & $E_{6,1}^4$                 &  312 & 12 & $1, 1, 3/4, 1, 1, 0$              & $\tilde{E}_6^4$\\
55 & A13 & $A_{9,1}^2D_{6,1}$          &  264 & 10 & $1, 1, 1, 0$                      & $\tilde{A}_9^2\tilde{D}_6$\\
54 & A14 & $D_{6,1}^4$                 &  264 & 10 & $1, 4/5, 1, 0$                    & $\tilde{D}_6^4$\\
51 & A15 & $A_{8,1}^3$                 &  240 &  9 & $1, 1, 0$                         & $\tilde{A}_8^3$\\
49 & A16 & $A_{7,1}^2D_{5,1}^2$        &  216 &  8 & $1, 1, 1, 0$                      & $\tilde{A}_7^2\tilde{D}_5^2$\\
46 & A17 & $A_{6,1}^4$                 &  192 &  7 & $1, 0$                            & $\tilde{A}_6^4$\\
43 & A18 & $A_{5,1}^4D_{4,1}$          &  168 &  6 & $1, 1, 1, 0$                      & $\tilde{A}_5^4\tilde{D}_4$\\
42 & A19 & $D_{4,1}^6$                 &  168 &  6 & $1, 2/3, 1, 0$                    & $\tilde{D}_4^6$\\
37 & A20 & $A_{4,1}^6$                 &  144 &  5 & $1, 0$                            & $\tilde{A}_4^6$\\
30 & A21 & $A_{3,1}^8$                 &  120 &  4 & $1, 1, 0$                         & $\tilde{A}_3^8$\\
24 & A22 & $A_{2,1}^{12}$              &   96 &  3 & $1, 0$                            & $\tilde{A}_2^{12}$\\
15 & A23 & $A_{1,1}^{24}$              &   72 &  2 & $1, 0$                            & $\tilde{A}_1^{24}$\\
 1 & A24 & $\C^{24}$                   &   24 &  1 & $0$                               & $\emptyset$\\\hline
\multicolumn{6}{c}{Rk. 16, cycle shape $1^82^8$}\\\hline\\[-1.2em]
62 &  B1 & $B_{8,1}E_{8,2}$            &  384 & 30 & $1, 14/15, 1, 1, 1, 1, 1, 0$      & $A_1\tilde{E}_8$\\
56 &  B2 & $B_{6,1}C_{10,1}$           &  288 & 22 & $1, 10/11, 1, 0$                  & $A_1A_9$\\
52 &  B3 & $C_{8,1}F_{4,1}^2$          &  240 & 18 & $1, 8/9, 1, 1, 1, 0$              & $A_2^2A_7$\\
53 &  B4 & $B_{5,1}E_{7,2}F_{4,1}$     &  240 & 18 & $1, 8/9, 1, 1, 1, 0$              & $A_1A_2\tilde{E}_7$\\
50 &  B5 & $A_{7,1}D_{9,2}$            &  216 & 16 & $1, 1, 1, 1, 0$                   & $\tilde{D}_9$\\
47 &  B6 & $B_{4,1}^2D_{8,2}$          &  192 & 14 & $1, 6/7, 1, 0$                    & $A_1^2\tilde{D}_8$\\
48 &  B7 & $B_{4,1}C_{6,1}^2$          &  192 & 14 & $1, 6/7, 1, 0$                    & $A_1A_5^2$\\
44 &  B8 & $A_{5,1}C_{5,1}E_{6,2}$     &  168 & 12 & $1, 1, 1, 1, 1, 0$                & $A_4\tilde{E}_6$\\
40 &  B9 & $A_{4,1}A_{9,2}B_{3,1}$     &  144 & 10 & $1, 1, 1, 0$                      & $A_1\tilde{A}_9$\\
39 & B10 & $B_{3,1}^2C_{4,1}D_{6,2}$   &  144 & 10 & $1, 4/5, 1, 0$                    & $A_1^2A_3\tilde{D}_6$\\
38 & B11 & $C_{4,1}^4$                 &  144 & 10 & $1, 4/5, 1, 0$                    & $A_3^4$\\
33 & B12 & $A_{3,1}A_{7,2}C_{3,1}^2$   &  120 &  8 & $1, 1, 1, 0$                      & $A_2^2\tilde{A}_7$\\
31 & B13 & $A_{3,1}^2D_{5,2}^2$        &  120 &  8 & $1, 1, 1, 0$                      & $\tilde{D}_5^2$\\
26 & B14 & $A_{2,1}^2A_{5,2}^2B_{2,1}$ &   96 &  6 & $1, 1, 1, 0$                      & $A_1\tilde{A}_5^2$\\
25 & B15 & $B_{2,1}^4D_{4,2}^2$        &   96 &  6 & $1, 2/3, 1, 0$                    & $A_1^4\tilde{D}_4^2$\\
16 & B16 & $A_{1,1}^4A_{3,2}^4$        &   72 &  4 & $1, 1, 0$                         & $\tilde{A}_3^4$\\
 5 & B17 & $A_{1,2}^{16}$              &   48 &  2 & $1, 0$                            & $\tilde{A}_1^{16}$\\
\end{tabular}
\renewcommand{\arraystretch}{1}
\label{table:70}
\end{table}
\addtocounter{table}{-1}

\begin{table}[p]
\caption{(continued)}
\renewcommand{\arraystretch}{1.2}
\begin{tabular}{rllrrll}
No. & No. & $(V_\Lambda^{\orb(g)})_1$ & Dim. & $n$ & $\rho(V_\Lambda(g^m))$ & $\Phi(g)$\\\hline\hline
\multicolumn{6}{c}{Rk. 12, cycle shape $1^63^6$}\\\hline
45 &  C1 & $A_{5,1}E_{7,3}$            & 168 & 18 & $1, 1, 1, 1, 1, 0$   & $\tilde{E}_7$\\
34 &  C2 & $A_{3,1}D_{7,3}G_{2,1}$     & 120 & 12 & $1, 1, 1, 1, 1, 0$   & $A_1\tilde{D}_7$\\
32 &  C3 & $E_{6,3}G_{2,1}^3$          & 120 & 12 & $1, 1, 3/4, 1, 1, 0$ & $A_1^3\tilde{E}_6$\\
27 &  C4 & $A_{2,1}^2A_{8,3}$          &  96 &  9 & $1, 1, 0$            & $\tilde{A}_8$\\
17 &  C5 & $A_{1,1}^3A_{5,3}D_{4,3}$   &  72 &  6 & $1, 1, 1, 0$         & $\tilde{A}_5\tilde{D}_4$\\
 6 &  C6 & $A_{2,3}^6$                 &  48 &  3 & $1, 0$               & $\tilde{A}_2^6$\\\hline
\multicolumn{6}{c}{Rk. 12, cycle shape $2^{12}$ (order doubling)}\\\hline
57 & D1a & $B_{12,2}$                  & 300 & 46 & $1, 22/23, 1, 0$     & $A_1$\\
41 & D1b & $B_{6,2}^2$                 & 156 & 22 & $1, 10/11, 1, 0$     & $A_1^2$\\
29 & D1c & $B_{4,2}^3$                 & 108 & 14 & $1, 6/7, 1, 0$       & $A_1^3$\\
23 & D1d & $B_{3,2}^4$                 &  84 & 10 & $1, 4/5, 1, 0$       & $A_1^4$\\
12 & D1e & $B_{2,2}^6$                 &  60 &  6 & $1, 2/3, 1, 0$       & $A_1^6$\\
 2 & D1f & $A_{1,4}^{12}$              &  36 &  2 & $1, 0$               & $\tilde{A}_1^{12}$\\
36 & D2a & $A_{8,2}F_{4,2}$            & 132 & 18 & $1, 1, 1, 1, 1, 0$   & $A_2$\\
22 & D2b & $A_{4,2}^2C_{4,2}$          &  84 & 10 & $1, 1, 1, 0$         & $A_3$\\
13 & D2c & $A_{2,2}^4D_{4,4}$          &  60 &  6 & $1, 1, 1, 0$         & $\tilde{D}_4$\\\hline
\multicolumn{6}{c}{Rk. 10, cycle shape $1^42^24^4$}\\\hline
35 & E1 & $A_{3,1}C_{7,2}$             & 120 & 16 & $1, 1, 1, 1, 0$      & $A_6$\\
28 & E2 & $A_{2,1}B_{2,1}E_{6,4}$      &  96 & 12 & $1, 1, 1, 1, 1, 0$   & $\tilde{E}_6$\\
18 & E3 & $A_{1,1}^3A_{7,4}$           &  72 &  8 & $1, 1, 1, 0$         & $\tilde{A}_7$\\
19 & E4 & $A_{1,1}^2C_{3,2}D_{5,4}$    &  72 &  8 & $1, 1, 1, 0$         & $A_2\tilde{D}_5$\\
 7 & E5 & $A_{1,2}A_{3,4}^3$           &  48 &  4 & $1, 1, 0$            & $\tilde{A}_3^3$\\\hline
\multicolumn{6}{c}{Rk. 8, cycle shape $1^45^4$}\\\hline
20 & F1 & $A_{1,1}^2D_{6,5}$           &  72 & 10 & $1, 1, 1, 0$         & $\tilde{D}_6$\\
 9 & F2 & $A_{4,5}^2$                  &  48 &  5 & $1, 0$               & $\tilde{A}_4^2$\\\hline
\multicolumn{6}{c}{Rk. 8, cycle shape $1^22^23^26^2$}\\\hline
21 & G1 & $A_{1,1}C_{5,3}G_{2,2}$      &  72 & 12 & $1, 1, 1, 1, 1, 0$   & $A_1A_4$\\
 8 & G2 & $A_{1,2}A_{5,6}B_{2,3}$      &  48 &  6 & $1, 1, 1, 0$         & $A_1\tilde{A}_5$\\\hline
\multicolumn{6}{c}{Rk. 6, cycle shape $1^37^3$}\\\hline
11 & H1 & $A_{6,7}$                    &  48 &  7 & $1, 0$               & $\tilde{A}_6$\\\hline
\multicolumn{6}{c}{Rk. 6, cycle shape $1^22^14^18^2$}\\\hline
10 & I1 & $A_{1,2}D_{5,8}$             &  48 &  8 & $1, 1, 1, 0$         & $\tilde{D}_5$\\\hline
\multicolumn{6}{c}{Rk. 6, cycle shape $2^36^3$ (order doubling)}\\\hline
14 & J1a & $A_{2,2}F_{4,6}$            &  60 & 18 & $1, 1, 1, 1, 1, 0$   & $A_2$\\
 3 & J1b & $A_{2,6}D_{4,12}$           &  36 &  6 & $1, 1, 1, 0$         & $\tilde{D}_4$\\\hline
\multicolumn{6}{c}{Rk. 4, cycle shape $2^210^2$ (order doubling)}\\\hline
4 & K1 & $C_{4,10}$                    &  36 & 10 & $1, 1, 1, 0$         & $A_3$
\end{tabular}
\renewcommand{\arraystretch}{1}
\end{table}

\FloatBarrier


\bibliographystyle{alpha_noseriescomma}
\bibliography{quellen}{}

\end{document}